\documentclass[12pt,a4paper,reqno]{amsart}
\usepackage[utf8]{inputenc}
\usepackage{enumerate}
\usepackage{mathtools}
\usepackage{mathrsfs}
\usepackage{soul, xcolor}
\usepackage{wasysym}
\usepackage{amsmath,amssymb,xpatch,relsize,graphics,graphicx,amsthm}
\usepackage{float,graphicx}
\usepackage{subcaption}

\xapptocmd\normalsize{%
	\abovedisplayskip=12pt plus 3pt minus 9pt
	\abovedisplayshortskip=0pt plus 3pt
	\belowdisplayskip=12pt plus 3pt minus 9pt
	\belowdisplayshortskip=7pt plus 3pt minus 4pt
}{}{}
\theoremstyle{definition}
\newtheorem{definition}{Definition}[section]
\newtheorem{remark}[definition]{Remark}

\theoremstyle{plain}
\newtheorem{theorem}[definition]{Theorem}

\numberwithin{equation}{section}

\title[Constrained Radius Estimates of certain analytic functions]{\textbf{Constrained Radius Estimates of certain analytic functions}}
\author[M. Sharma]{Meghna Sharma}
\address{Department of Mathematics, University of Delhi, Delhi--110 007, India}
\email{meghnasharma203@gmail.com}
\author[N.K. Jain]{Naveen Kumar Jain}
\address {Department of Mathematics, Aryabhatta College, Delhi-110021,India}
\email{naveenjain05@gmail.com}
\author[S. Kumar]{Sushil Kumar}
\address {Bharati Vidyapeeth's college of Engineering, Delhi-110063, India}
\email{sushilkumar16n@gmail.com}

\date{}



\thanks{The first author is supported by Senior Research Fellowship from Council of Scientific and Industrial Research, New Delhi, Ref. No.:1753/(CSIR-UGC NET JUNE, 2018).}

\allowdisplaybreaks

\begin{document}
\maketitle
\begin{abstract}
Let $\mathcal{P}$ denote the Carath\'{e}odory class accommodating all the analytic functions $p$ having positive real part and satisfying $p(0)=1$.
In this paper, the second coefficient of the normalized analytic function $f$	defined on the open unit disc is constrained to define new classes of analytic functions.
The classes are characterised by the functions $f/g$ having positive real part or satisfying the inequality $|(f(z)/g(z))-1|<1$ such that
$f(z)(1-z^2)/z$ and $g(z)(1-z^2)/z$ are Carath\'{e}odory functions for some analytic function $g$.
This paper aims at determining radius of starlikeness for the introduced classes.
\end{abstract}

\section{Classes of Starlike Functions}
Denote by $\mathcal{A}$ the class of all analytic functions $f: \mathbb{D} \to \mathbb{C}$ defined on $\mathbb{D}:=\{z \in \mathbb{C}: |z|<1\}$ and satisfying the normalizing conditions $f(0)=f'(0)-1=0$.
Let $\mathcal{S}$ be the subclass of $\mathcal{A}$ consisting of univalent functions.
The function $p \in \mathcal{A}$ with $p(0)=1$ and having positive real part is labeled as a Carath\'{e}odory function, denoted by $\mathcal{P}$.
Given two subclasses $\mathcal{M}$ and $\mathcal{N}$ of $\mathcal{A}$, the largest number $\rho \in (0,1)$ is the $\mathcal{M}$- radius for the class $\mathcal{N}$ if for all $f \in \mathcal{N}$, $r^{-1}f(rz)$ belongs to $\mathcal{M}$, where $0<r \leq \rho$. Thus, the radius problem is essentially finding the desired value of $\rho$. For more details, see \cite{MR0708494}.

Proceedings of a conference from 1992 revealed that Ma and Minda \cite{MR1343506} consolidated various concepts of subordination to introduce a unified subclass $S^{*}(\varphi)$ of starlike functions where $\varphi \in \mathcal{P}$, is starlike with respect to $\varphi(0)=1$, symmetric about  real axis and satisfy $\varphi'(0)>0$.
This class is associated with the quantity $zf'(z)/f(z)$ subordinate to the function $\varphi$.
Assigning a particular value to the function $\varphi(z)$ yields various subclasses of $S^{*}$.
Essentially, $S^{*}(\varphi)$ is reduced to the class of Janowski starlike functions when $\varphi(z)=(1+Az)/(1+Bz)$ for $-1 \leq B < A \leq 1$.
Introduced by Robertson \cite{MR0783568}, $A=1-2\alpha$ ($0 \leq \alpha <1$) and $B=-1$ gives the subclass $S^{*}(\alpha)$ of starlike functions of order $\alpha$.
Following the similar approach, several authors established various subclasses of starlike functions.
Sok\'{o}\l\ and Stankiewicz \cite{MR1473947} introduced the class $S^{*}_{L}=S^{*}(\sqrt{1+z})$ associated with the region bounded by the right half of the lemniscate of bernoulli $\{w \in \mathbb{C}: |w^2 - 1|=1\}$,
Ronning \cite{MR1128729} introduced the class $S^{*}_{P}=S^{*}(1+(2/\pi^2)(\log((1+\sqrt{z})/(1-\sqrt{z})))^2)$ associated with parabolic region $\{u+i v: 2u-1>v^2\}$,
Mendiratta \emph{et al.} \cite{MR3394060} introduced the class $S^{*}_{e}=S^{*}(e^z)$ associated with the domain $\{w \in \mathbb{C}: |\log w|<1\}$,
Sharma \emph{et al.} \cite{MR3536076} introduced the class $S^{*}_{c}=S^{*}(1+4z/3+2z^2/3)$ associated with the cardioid shaped region $\{u+ i v : (9u^2+9v^2-18u+5)^2 - 16(9u^2+9v^2-6u+1)=0\}$,
Cho \emph{et al.} \cite{MR3913990} introduced the class $S^{*}_{\sin}=S^{*}(1+\sin z)$ associated with the sine function,
Raina and Sok\'{o}\l\ \cite{MR3419845} introduced the class $S^{*}_{\leftmoon}=S^{*}(z+\sqrt{1+z^2})$ associated with the lune $\{w \in \mathbb{C}:\operatorname{Re} w > 0, 2|w|>|w^2-1|\}$,
Kumar and Ravichandran \cite{MR3496681} introduced the class $S^{*}_{R}=S^{*}(1+(z/k)((k+z)/(k-z)))$, $k=1+\sqrt{2}$, associated with the rational function,
Brannan and Kirwan \cite{MR0251208} introduced the class $S^{*}_{\gamma}=S^{*}((1+z)/(1-z)^{\gamma})$, $0 \leq \gamma <1$, of strongly starlike functions of order $\gamma$.
Recently, authors \cite{MR4190740} introduced the class $S^{*}_{Ne}=S^{*}(1+z-z^3/3)$ associated with the nephroid domain $\{u+i v:((u-1)^2+v^2-4/9)^3-4v^2/3=0\}$.
Also, the class $S^{*}_{SG}=S^{*}(2/(1+e^{-z}))$ associated with modified sigmoid domain $\{w \in \mathbb{C}:|\log (w/ (2-w))|<1\}$ is introduced in \cite{MR4044913}.

\section{ Classes of Analytic Functions}

Owing to the eminent Bieberbach theorem, the estimate on the second coefficient in Maclaurin series of the function $f \in \mathcal{A}$ plays a vital role in the study of univalent functions.
We express $\mathcal{A}_{b}$ as the class of all functions $f \in \mathcal{A}$ having the form $f(z)=z+a_2z^2+\cdots$, where $|a_2|=2b$ for $0 \leq b \leq 1$.
Let $\mathcal{P}(\alpha)$ denote the class of analytic functions of the form
$p(z)=1+a_1 z+a_2 z^2+\cdots$
with $\operatorname{Re}(p(z))> \alpha$, $0 \leq \alpha <1$.
According to a result by Nehari \cite{MR0377031}, we have $|a_n| \leq 2(1-\alpha)$ for $p\in \mathcal{P}(\alpha)$.
Thus, we considered a subclass $\mathcal{P}_{b}(\alpha)$ of $\mathcal{P}(\alpha)$ having functions of the form
\[p(z)=1+2b(1-\alpha)z+a_2 z^2+\cdots, \quad |b| \leq 1.\]
Gronwall initiated the study of radius problems for the functions with
fixed second coefficient in early 1920s and since then, this aspect has been an active area of research.
Further, MacGregor \cite{MR0148892,Mcgregor} studied the radius problems for the class of functions having either the positive real part of the ratio $f/g$ or satisfying the inequality $|f/g-1|<1$.
Ali \emph{et al.} \cite{MR3722703} also made a contribution by finding various radius constants involving second order differential subordination.
Recently, authors \cite{MR4166151} studied radius problems on the class of
functions involving the ratio $f/g$.
For literature related to the applications of differential subordination for
functions with fixed second coefficient, see [add citation.]

Inspired by the above mentioned work and taking in account of the coefficient bound, we restricted the second coefficient of the function and introduced certain subclasses of analytic functions.

For $n \in \mathbb{N}$, let us consider the analytic functions of the form
$f(z)=z+\sum\limits_{n=2}^{\infty}a_{n}z^n$
such $f(z)(1-z^2)/z \in \mathcal{P}$.
Taking in note of the Bieberbach conjecture, we conclude that the coefficient $a_2$ is bounded by 2 and subsequently, the function $f(z)$ can be rewritten as
\[f(z)=z+2bz^2+\sum\limits_{n=3}^{\infty}a_{n}z^n, \quad |b| \leq 1.\]

Thenceforth, we now define the class $\mathcal{K}^{1}_{b}$ as follows:
\begin{definition}
For $|b| \leq 1$, the class $\mathcal{K}^{1}_{b}$ is defined as
\[\mathcal{K}^{1}_{b}:=\left\{f(z)=z+2bz^2+\sum\limits_{n=3}^{\infty}a_{n}z^n: \operatorname{Re}\left(\frac{f(z)(1-z^2)}{z}\right)>0, z \in \mathbb{D}\right\}.\]
\end{definition}
Consider the function $f_{b}: \mathbb{D} \to \mathbb{C}$ defined by
\begin{equation}\label{function-k3}
f_{b}(z)=\frac{z \left(1+z^2\right)}{\left(1-z^2\right) \left(1-2 b i z-z^2\right)}
\end{equation}
with $s_1(z)=z(b-iz)/(i+bz)$.
Then, it can be easily seen that
\[\frac{f_{b}(z)(1-z^2)}{z}=\frac{1-s_1(z)}{1+s_1(z)},\]
where $s_1(z)$ is the analytic function satisfying the hypothesis of Schwarz's Lemma in $\mathbb{D}$ and yields $\operatorname{Re}(f_{b}(z)(1-z^2)/z)>0$.
Hence, the class $\mathcal{K}^{1}_{b}$ is non empty.

Further, in order to define another class of analytic functions, consider the analytic functions of the form
\begin{equation}\label{fandg}
f(z)=z+\sum\limits_{n=2}^{\infty}f_{n}z^n \quad \text{and} \quad g(z)=z+\sum\limits_{n=2}^{\infty}g_{n}z^n
\end{equation}
such that $f(z)/g(z) \in \mathcal{P}$ and $g(z)(1-z^2)/z \in \mathcal{P}$.
Now, using these conditions, we get that the coefficients $f_2$ and $g_2$ satisfies $|f_2| < 4$ and $|g_2| \leq 2$.
Thus, the functions $f$ and $g$ are rewritten as
\[f(z)=z+4bz^2+\sum\limits_{n=3}^{\infty}f_{n}z^n, |b| \leq 1 \quad \text{and} \quad g(z)=z+2cz^2+\sum\limits_{n=3}^{\infty}g_{n}z^n, |c| \leq 1.\]
\begin{definition}
For $|b| \leq 1$ and $|c| \leq 1$, the class $\mathcal{K}^{2}_{b,c}$ is defined as
\[\mathcal{K}^{2}_{b,c}:=\left\{f \in \mathcal{A}_{4b}:\operatorname{Re}\left(\frac{f(z)}{g(z)}\right)>0, \operatorname{Re}\left(\frac{g(z)(1-z^2)}{z}\right)>0; \text{ for some } g \in \mathcal{A}_{2c}\right\}.\]
\end{definition}
Note that the functions $f_{b,c}, g_{b,c}: \mathbb{D} \to \mathbb{C}$ defined by
\begin{equation}\label{f1}
f_{b,c}(z)=\frac{z \left(1+z^2\right)^2}{\left(1-z^2\right) \left(1-2 c i z-z^2\right) \left(1-(4 b-2 c)i z - z^2\right)}
\end{equation}
and
\[g_{b,c}(z)=\frac{z \left(1+z^2\right)}{\left(1-z^2\right) \left(1-2 c i z - z^2\right)}\]
satisfy
\[\frac{f_{b,c}(z)}{g_{b,c}(z)}=\frac{1-s_2(z)}{1+s_2(z)} \quad \text{and} \quad  \frac{g_{b,c}(z)(1-z^2)}{z}=\frac{1-s_3(z)}{1+s_3(z)},\]
where \[s_2(z)=\frac{z((2b-c)-iz)}{(2b-c)z+i} \quad \text{and} \quad s_3(z)=\frac{z \left(c- i z\right)}{cz+i}, \quad |2b-c| \leq 1.\]
Note that $s_2$ and $s_3$ are analytic functions and satisfy the hypothesis of Schwarz's lemma in $\mathbb{D}$ and therefore, $\operatorname{Re}\left(\frac{f_{b,c}(z)}{g_{b,c}(z)}\right) > 0$ and $\operatorname{Re}\left(\frac{g_{b,c}(z)(1-z^2)}{z}\right) >0$.
Thus, $f_{b,c}$ and $g_{b,c}$ are the members of $\mathcal{K}^{2}_{b,c}$ and hence, the class is non-empty.

Following the similar pattern, let us assume that the functions $f$ and $g$ as defined in (\ref{fandg}) and satisfying the inequalities $|(f(z)/g(z))-1|<1$ and $\operatorname{Re}((g(z)(1-z^2))/z)>0$.
Note that the condition $|(f(z)/g(z))-1|<1$ yields $\operatorname{Re}(g(z)/f(z))>1/2$, thereby implying $|f_2| \leq 1+|g_2| \leq 3$.
Consequently, we consider the functions of the following form:
\[f(z)=z+3bz^2+\sum\limits_{n=3}^{\infty}f_{n}z^n, |b| \leq 1 \quad \text{and} \quad g(z)=z+2cz^2+\sum\limits_{n=3}^{\infty}g_{n}z^n, |c| \leq 1.\]

\begin{definition}
For $|b| \leq 1$ and $|c| \leq 1$, the class $\mathcal{K}^{3}_{b,c}$ is defined as
\[\mathcal{K}^{3}_{b,c}:=\left\{f \in \mathcal{A}_{3b}: \left|\frac{f(z)}{g(z)}-1\right| < 1 \text{ and } \operatorname{Re}\left(\frac{g(z)(1-z^2)}{z}\right)>0; \text{ for some } g \in \mathcal{A}_{2c}\right\}.\]
\end{definition}


\begin{remark}
For $b=-1$ and $c=-1$, the classes $\mathcal{K}_b^{1}$, $\mathcal{K}^{2}_{b,c}$ and $\mathcal{K}^{3}_{b,c}$ coincide with the classes $\mathcal{K}_3$, $\mathcal{K}_1$ and $\mathcal{K}_2$ respectively as discussed in \cite{MR4166151}.
\end{remark}

\section{Constraint Radius Estimates}

We constrained the second coefficient of the functions $f$ and $g$. The radii estimates are calculated for the classes involving these constraints.
For several recent results involving radius problems with fixed second coefficient, see[\cite{MR3531955},]

Various radius problems for these classes were discussed and sharpness was proved analytically.
Radii estimates for the class $\mathcal{K}^{1}_{b}$ are obtained in the following theorem.

\begin{theorem}\label{Theorem 1}
Let $\tilde{n}=|2b|$. For the class $\mathcal{K}^{1}_{b}$, the following results hold.
\begin{enumerate}[(i)]
\item The $S^{*}_{P}$ radius is the smallest positive real root of the equation
\begin{align}
1-\tilde{n} r-11 r^2-8 \tilde{n} r^3-9 r^4+\tilde{n} r^5+3 r^6 = 0. \label{K3-S-P}
\end{align}

\item  Let $\alpha \in [0,1)$. The $S^{*}(\alpha)$ radius is the smallest positive real root of the equation
\begin{align}
1-\alpha-\alpha \tilde{n} r -r^2 (5+\alpha)-4 \tilde{n} r^3-(5-\alpha)r^4 +\alpha \tilde{n}r^5 +(1+\alpha)r^6=0. \label{K3-S-alpha}
\end{align}
		
\item The $S^{*}_{L}$ radius is the smallest positive real root of the equation
\begin{align}
1-\sqrt{2}+(2 \tilde{n}-\sqrt{2} \tilde{n})r+6 r^2+(2 \tilde{n}+\sqrt{2} \tilde{n})r^3+(1+\sqrt{2})r^4=0.  \label{K3-S-L}
\end{align}
		
\item The $S^{*}_{e}$ radius is the smallest positive real root of the equation
\begin{align}
1-e+\tilde{n} r+(1+5 e) r^2+4 e \tilde{n} r^3-(1-5 e) r^4-\tilde{n} r^5-(1+e) r^6 = 0. \label{K3-S-e}
\end{align}
		
\item The $S^{*}_{c}$ radius is the smallest positive real root of the equation
\begin{align}
2-\tilde{n} r-16 r^2-12 \tilde{n} r^3-14 r^4+\tilde{n} r^5+4 r^6 = 0. \label{K3-S-c}
\end{align}
		
\item The $S^{*}_{\sin}$ radius is the smallest positive real root of the equation
\begin{align}
&\sin 1-(\tilde{n}-\tilde{n} \sin 1)r-(6-\sin 1)r^2-4 \tilde{n} r^3-(8+\sin 1)r^4-(3 \tilde{n}+\tilde{n} \sin 1)r^5 \nonumber \\
&-(2+\sin 1)r^6 = 0. \label{K3-S-sin}
\end{align}
		
\item The $S^{*}_{\leftmoon}$ radius is the smallest positive real root of the equation
\begin{align}
2-\sqrt{2}+(\tilde{n}-\sqrt{2} \tilde{n}) r-(4+\sqrt{2}) r^2-4 \tilde{n} r^3-(6-\sqrt{2}) r^4-(\tilde{n}-\sqrt{2} \tilde{n}) r^5+\sqrt{2} r^6 = 0. \label{K3-S-lune}
\end{align}
		
\item The $S^{*}_{R}$ radius is the smallest positive real root of the equation
\begin{align}
&3-2 \sqrt{2}+(2 \tilde{n}-2 \sqrt{2} \tilde{n}) r-(3+2 \sqrt{2}) r^2-4 \tilde{n} r^3-(7-2 \sqrt{2}) r^4-(2 \tilde{n}-2 \sqrt{2} \tilde{n}) r^5 \nonumber \\
&-(1-2 \sqrt{2}) r^6 = 0. \label{K3-S-R}
\end{align}
		
		
\item The $S^{*}_{N_{e}}$ radius is the smallest positive real root of the equation
\begin{align}
2-\tilde{n} r-16 r^2-12 \tilde{n} r^3-26 r^4-11 \tilde{n} r^5-8 r^6 = 0.  \label{K3-S-Ne}
\end{align}
		
\item The $S^{*}_{SG}$ radius is the smallest positive real root of the equation
\begin{align}
1-e+2 \tilde{n}r+(7+5 e)r^2+ 4\tilde{n}(1+e)r^3+(7+9 e) r^4+2 \tilde{n}(1+2e)r^5+(1+3 e) r^6 = 0.  \label{K3-S-SG}
\end{align}
\end{enumerate}
All estimates are sharp.
\end{theorem}

\begin{proof}
For the function $p \in \mathcal{P}(\alpha)$ and $|z|=r<1$, by \cite[Theorem 2]{MR0298014}, we have
\begin{equation}\label{mainlemma}
\left|\frac{zp'(z)}{p(z)}\right| \leq \frac{2(1-\alpha)r}{1-r^2}\frac{|b|r^2+2r+|b|}{(1-2\alpha) r^2 +2|b|(1-\alpha)r +1}
\end{equation}
where $|b| \leq 1$ and $\alpha \in [0,1)$.
The transform $w(z)=\frac{1+z^2}{1-z^2}$ maps the disc $|z| \leq r$ onto the disc
\begin{equation}\label{transformmap}
\left|w(z)-\frac{1+r^4}{1-r^4}\right| \leq \frac{2r^2}{1-r^4}.
\end{equation}

Suppose $f \in \mathcal{K}^{1}_{b}$ and let $p:\mathbb{D} \to \mathbb{C}$ be given by $p(z)=f(z)(1-z^2)/z$.
Note that the function $p \in \mathcal{P}_{b}$ and a simple calculation gives
\begin{equation}\label{logdiff-K3}
\frac{zf'(z)}{f(z)}=\frac{zp'(z)}{p(z)}+\frac{1+z^2}{1-z^2}.
\end{equation}
Using (\ref{mainlemma}), (\ref{transformmap}) and (\ref{logdiff-K3}), it is seen that $f$ maps the disc $|z| \leq r$ onto the disc
\begin{equation}\label{discforK3}
\left|\frac{zf'(z)}{f(z)}-\frac{1+r^4}{1-r^4}\right| \leq
\frac{\tilde{n} r+6 r^2+4 \tilde{n} r^3+6 r^4+\tilde{n} r^5}{(1 + \tilde{n} r + r^2) (1 - r^4)}.
\end{equation}
	
\begin{enumerate}[(i)]
\item Set $x(r):=1-\tilde{n} r-11 r^2-8 \tilde{n} r^3-9 r^4+\tilde{n} r^5+3 r^6.$
Then, $x(0)=1>0$ and
$x(1)=-8(\tilde{n}+2)<0$.
Hence by the virtue of intermediate value theorem, the equation (\ref{K3-S-P}) has a root lying in the interval $(0,1)$, denoted by $\rho_1$.

Denote the center of the disc in (\ref{discforK3}) by $a$.
Clearly, $a \geq 1$ for $r \in [0,1)$ and $a \leq 3/2$ if $r < 1/5^{\frac{1}{4}} \approx 0.66874$.
Thus, an application of \cite[Lemma 2.2]{MR1415180} gives that the disc (\ref{discforK3}) lies in the region $\{w \in \mathbb{C}: |w-1| < \operatorname{Re}w\}$ if
\[\frac{\tilde{n} r+6 r^2+4 \tilde{n} r^3+6 r^4+ \tilde{n} r^5}{(1 + \tilde{n} r + r^2) (1 - r^4)} \leq \frac{1+r^4}{1-r^4}-\frac{1}{2}.\]
Or equivalently, if
\[\frac{1-5r^2+8br^3-5r^4+r^6}{(-1+2 b r-r^2)(1-r^4)} \leq -\frac{1}{2}.\]
Thus,
\[\operatorname{Re}\left(\frac{zf'(z)}{f(z)}\right) > \left|\frac{zf'(z)}{f(z)}-1\right|\]
for $0<r\leq \rho_1$ proving that the number $\rho_1$ is the $S^{*}_{P}$ radius for the class $\mathcal{K}^{1}_{b}$.

For $b<0$, the function defined in (\ref{function-k3}), for $z=- i \rho_1$ satisfies
\begin{align*}
\operatorname{Re}\left(\frac{z(f_{b})'(z)}{f_{b}(z)}\right)
&=\frac{1 - 5 \rho_1^2 + 8 b \rho_1^3 - 5 \rho_1^4 + \rho_1^6}{(1 - 2 b \rho_1 + \rho_1^2) (1 - \rho_1^4)}\\
&=\left|\frac{2 b \rho_1-6 \rho_1^2+8 b \rho_1^3-4 \rho_1^4-2 b \rho_1^5+2 \rho_1^6}{\left(1-2 b \rho_1+\rho_1^2\right) \left(1-\rho_1^4\right)}\right|\\
&=\left|\frac{z(f_{b})'(z)}{f_{b}(z)}-1\right|
\end{align*}
illustrating the sharpness of bound.

\begin{remark}
Throughout the paper, we have proved the sharpness when $b<0$. The method can be imitated to prove for $b>0$. Hence, sharpness is attained at the point
\[z = \left\{
\begin{array}{ll}
-i \rho, & \text{if } b<0\\
i \rho,  & \text{if } b>0\\
\end{array}
\right.\]
for the classes $S^{*}_{P}$, $S^{*}(\alpha)$, $S^{*}_{e}$, $S^{*}_{c}$, $S^{*}_{\leftmoon}$, $S^{*}_{R}$ and for the classes $S^{*}_{L}$, $S^{*}_{\sin}$, $S^{*}_{Ne}$ and $S^{*}_{SG}$, sharpness is attained at the point
\[z = \left\{
\begin{array}{ll}
	 \rho, & \text{if } b<0\\
	- \rho, & \text{if } b>0.\\
\end{array}
\right.\]
\end{remark}

\begin{figure}[h!]
	\centering
	\begin{subfigure}[!]{0.11\linewidth}
		\includegraphics[width=\linewidth]{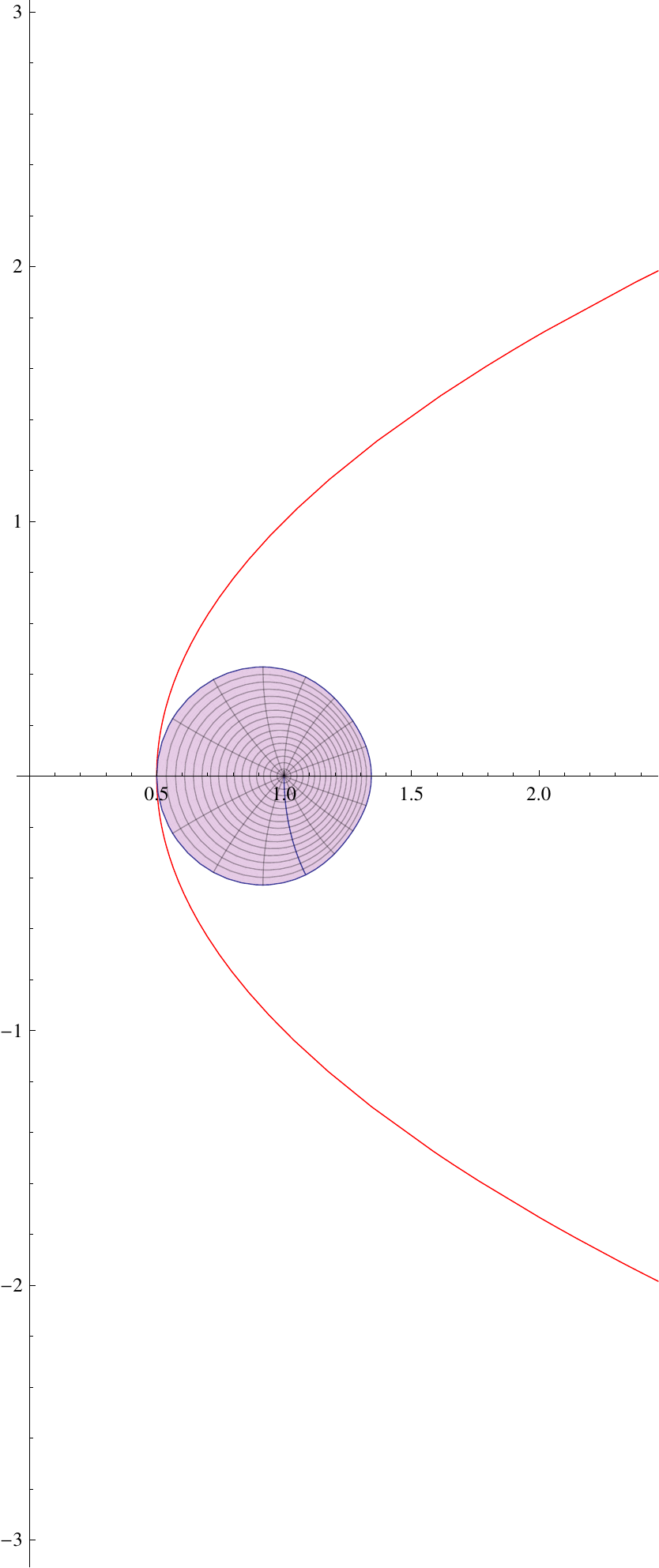}
		\caption{Sharpness of class $S_{P}^{*}$ with $\rho_{1}=0.2021347$ at $b=-1$}	\end{subfigure}
	\hspace{2.4em}
	\begin{subfigure}[!]{0.20\linewidth}
		\includegraphics[width=\linewidth]{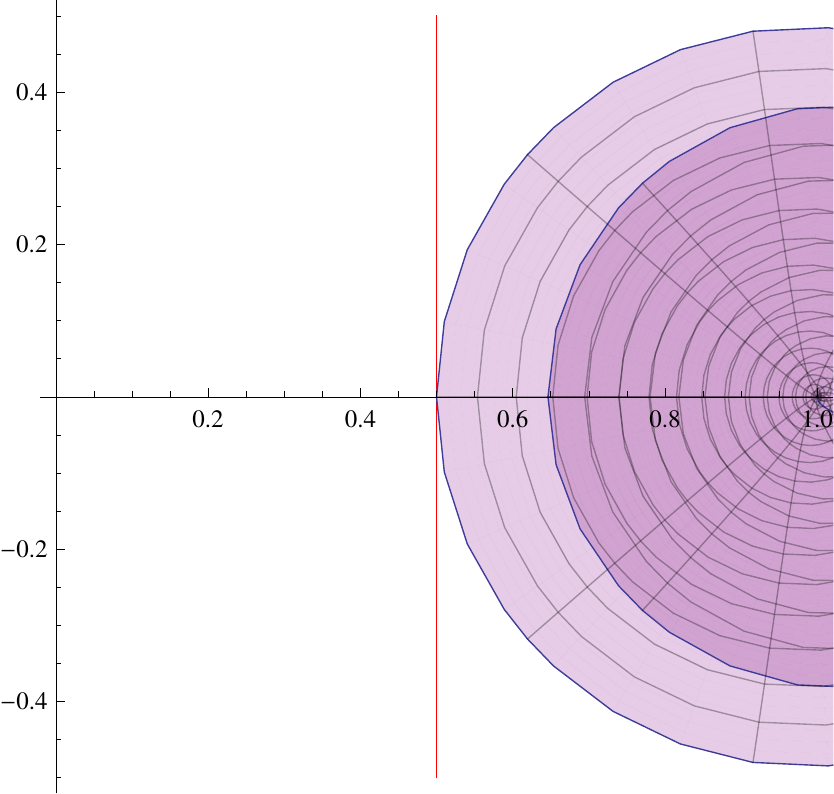}
		\caption{Sharpness of class $S^{*}(\alpha)$ with $\rho_{2}=0.202135$ at $b=-1$, $\alpha=0.5$}
	\end{subfigure}
    \hspace{2.4em}
    \begin{subfigure}[!]{0.21\linewidth}
	\includegraphics[width=\linewidth]{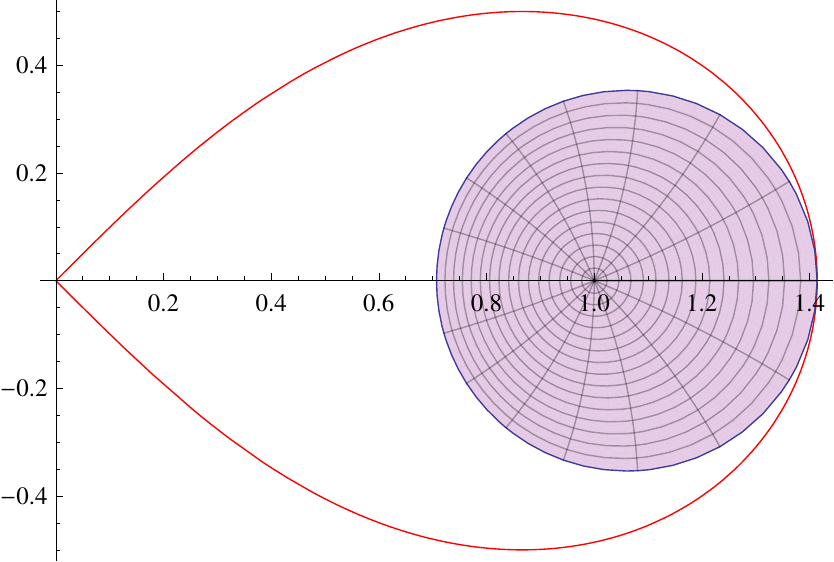}
	\caption{Sharpness of class $S^{*}_L$ with $\rho_3=0.171573$ at $b=-1$}
     \end{subfigure}
	\caption{Graphical illustration of sharpness for classes in Theorem (\ref{Theorem 1}) for particular choices of $b$.}
	\label{fig:coffee}
\end{figure}
	
\item Let $\alpha \in [0,1)$ and
let $\rho_2 \in (0,1)$ be the smallest positive real root of the equation (\ref{K3-S-alpha}) on account of the intermediate value theorem.
The function
\[d(r)=\frac{1 - 5 r^2 + 4 \tilde{n} r^3 - 5 r^4 + r^6}{(1 + \tilde{n} r + r^2) (1 - r^4)}\]
is a non-increasing function on $[0,1)$ and thus from (\ref{discforK3}), it follows that for $0<r \leq \rho_{2}$, we have,
\begin{align*}
\operatorname{Re}\left(\frac{zf'(z)}{f(z)}\right)
\geq \frac{1 - 5 r^2 + 4 \tilde{n} r^3 - 5 r^4 + r^6}{(1 + \tilde{n} r + r^2) (1 - r^4)}
\geq \alpha.
\end{align*}
This shows that the number $\rho_2$ is the $S^{*}(\alpha)$ radius for the class $\mathcal{K}^{1}_{b}$.
Since, for $f_{b}$ defined in (\ref{function-k3}) at $z=-i\rho_2$, we have
\[\operatorname{Re}\left(\frac{z(f_{b})'(z)}{f_{b}(z)}\right)
=\frac{1 - 5 \rho_2^2 + 8 b \rho_2^3 - 5 \rho_2^4 + \rho_2^6}{(1 - 2 b \rho_2 + \rho_2^2) (1 - \rho_2^4)}
=\alpha,\]
it follows that the estimate is sharp.

\item From (\ref{discforK3}), it follows that
\[
\left|\frac{zf'(z)}{f(z)}-1\right| \leq  \left|\frac{zf'(z)}{f(z)}-\frac{1+r^4}{1-r^4}\right|+ \frac{2 r^4}{1-r^4}
\leq \frac{r (\tilde{n} + 3 \tilde{n} r^2 + 2 r (3 + r^2))}{(1 - r^2) (1 + \tilde{n} r + r^2)}.
\]
Suppose $\rho_3$ denote the smallest real root of the equation in (\ref{K3-S-L}).
It is clear that $a \geq 1$ for $r \in [0,1)$ and simple computations show that $a < \sqrt{2}$ if $r < ((\sqrt{2}-1)/(\sqrt{2}+1))^{1/4} \approx 0.643594$.
Hence, using \cite[Lemma 2.2]{MR2879136}, the disc (\ref{discforK3}) lies in the region $\{w \in \mathbb{C}: |w^2-1| < 1\}$ provided
\[\frac{r (\tilde{n} + 3 \tilde{n} r^2 + 2 r (3 + r^2))}{(1 - r^2) (1 + \tilde{n} r + r^2)} \leq \sqrt{2}-1\]
Since, for $0< r \leq \rho_3$, we have
\[\left|\left(\frac{zf'(z)}{f(z)}\right)^2-1\right|<1.\]
This proves that the number $\rho_3$ is the $S^{*}_{L}$ radius for the class $\mathcal{K}^{1}_{b}$.
The function $f_{b}$ defined by
\begin{equation}\label{K3-SL-sharp}
f_{b}=-\frac{z (1-2 b z+z^2)}{(-1+z^2)^2}.
\end{equation}
is in the class $\mathcal{K}^1_{b}$ because it satisfies the condition $(1-z^2)f_{b}/z=(1-s(z))/(1+s(z))$, where $s(z)=(1-bz)/z(b-z)$.
Thus, for $f_{b}$ defined by (\ref{K3-SL-sharp}) at $z=\rho_{3}$, we have,
\begin{align*}
\left|\left(\frac{z(f_{b})'(z)}{f_{b}(z)}\right)^2-1\right|
&=\left|\left(\frac{1-4 b\rho_{3} +6 \rho_{3}^2-4 b \rho_{3}^3+\rho_{3}^4}{1-2 b \rho_{3}+2 b \rho_{3}^3-\rho_{3}^4}\right)^2-1\right|\\
&=|(\sqrt{2})^2-1|=1
\end{align*}
proving the sharpness.

\item The number $\rho_4$ is the smallest real root of the equation in (\ref{K3-S-e}).
Easy computations show that for $r \in [0,1)$, we have $a<e$ for $r< ((e-1)/(e+1))^{1/4} \approx 0.824495$.
Also, $a < (e+1/e)/2$ for $r<((e-1)/(e+1))^{1/2} \approx 0.679792$.
Thus, using \cite[Lemma 2.2]{MR3394060}, the disc (\ref{discforK3}) is contained in the region $\{w \in \mathbb{C}: |\log w| < 1\}$ provided
\[\frac{\tilde{n} r+6 r^2+4 \tilde{n} r^3+6 r^4+\tilde{n} r^5}{(1 + \tilde{n} r + r^2) (1 - r^4)} \leq \frac{1+r^4}{1-r^4}-\frac{1}{e}.\]
Or equivalently, if
\[\frac{1-5r^2+8br^3-5r^4+r^6}{(-1+2 b r-r^2)(1-r^4)} \leq -\frac{1}{e}\]
for $0<r\leq \rho_4$.
Thus, proving that the number $\rho_4$ is the $S^{*}_{e}$ radius for the class $\mathcal{K}^{1}_{b}$.
Moreover, for the function in (\ref{function-k3}) at $z=-i \rho_4$, we have,
\[\left|\log\left(\frac{z(f_{b})'(z)}{f_{b}(z)}\right)\right|=
\left|\log\left(\frac{1 - 5 \rho_4^2 + 8 b \rho_4^3 - 5 \rho_4^4 + \rho_4^6}{(1 - 2 b \rho_4 + \rho_4^2) (1 - \rho_4^4)}\right)\right|
=\left|\log\left(\frac{1}{e}\right)\right|=1,\]
showing that the bound is best possible.
	
\item The number $\rho_5$ is the smallest real root of the equation in (\ref{K3-S-c}).
For $r \leq 1/\sqrt{2} \approx 0.707107$, we have $a \leq 5/3$.
By \cite[Lemma 2.5]{MR3536076}, the disc (\ref{discforK3}) lies inside the region bounded by $\phi_{c}(\mathbb{D})$, where $\phi_{c}(z)=1+(4/3)z+(2/3)z^2$ if
\[\frac{\tilde{n} r+6 r^2+4 \tilde{n} r^3+6 r^4+\tilde{n} r^5}{(1 + \tilde{n} r + r^2) (1 - r^4)} \leq \frac{1+r^4}{1-r^4}-\frac{1}{3}.\]
Or equivalently, if
\[\frac{1-5r^2+8br^3-5r^4+r^6}{(-1+2 b r-r^2)(1-r^4)} \leq -\frac{1}{3}.\]
Hence, $f \in S^{*}_{c}$ for $0<r\leq \rho_5$ showing that the number $\rho_5$ is the $S^{*}_c$ radius for the class $\mathcal{K}^{1}_{b}$.
Further, for the function in (\ref{function-k3}) at $z=-i \rho_{5}$, we have,
\begin{align*}
\frac{z(f_{b})'(z)}{f_{b}(z)}
&=\frac{1 - 5 \rho_5^2 + 8 b \rho_5^3 - 5 \rho_5^4 + \rho_5^6}{(1 - 2 b \rho_5 + \rho_5^2) (1 - \rho_5^4)}\\
&=\frac{1}{3}= \phi_c(-1).
\end{align*}
Hence, proving the sharpness.
		
\item The number $\rho_6$ is the smallest real root of the equation in (\ref{K3-S-sin}).
When $1-\sin 1 < a \leq 1+\sin 1$, in view of \cite[Lemma 3.3]{MR3913990}, the function $f \in S^{*}_{\sin}$ if
\[\frac{\tilde{n} r+6 r^2+4 \tilde{n} r^3+6 r^4+\tilde{n} r^5}{(1 + \tilde{n} r + r^2) (1 - r^4)} \leq \sin 1 -\frac{2r^4}{1-r^4}.\]
Therefore, the disc (\ref{discforK3}) lies inside the region $\phi_{s}(\mathbb{D})$, where $\phi_{s}(z)=1+\sin z$ provided $0 < r \leq \rho_6$.
For the function in (\ref{K3-SL-sharp}) at $z=\rho_6$, we have
\begin{align*}
\frac{z(f_{b})'(z)}{f_{b}(z)}
&=\frac{1 - 5 \rho_6^2 + 8 b \rho_6^3 - 5 \rho_6^4 + \rho_6^6}{(1 - 2 b \rho_6 + \rho_6^2) (1 - \rho_6^4)}\\
&=1+\sin 1= \phi_s(1).
\end{align*}
Hence, the estimate is sharp.

\item  The number $\rho_7$ is the smallest real root of the equation in (\ref{K3-S-lune}).
For $\sqrt{2}-1 < a < \sqrt{2}+1$, applying \cite[Lemma 2.1]{MR3718233},
the disc (\ref{discforK3}) is contained in the region $\{w \in \mathbb{C}: 2|w| > |w^2-1|\}$ if
\[\frac{\tilde{n} r+6 r^2+4 \tilde{n} r^3+6 r^4+\tilde{n} r^5}{(1 + \tilde{n} r + r^2) (1 - r^4)} \leq \frac{1+r^4}{1-r^4}+1-\sqrt{2}.\]
Since, for $0<r \leq \rho_7$, we have
\[2\left|\frac{zf'(z)}{f(z)}\right| > \left|\left(\frac{zf'(z)}{f(z)}\right)^2-1\right|.\]
For sharpness, note that the function $f_{b}$ given in (\ref{function-k3}) at $z=-i \rho_7$ satisfies
\begin{align*}
\left|\left(\frac{z(f_{b})'(z)}{f_{b}(z)}\right)^2-1\right|
&=\left|\left(\frac{1 - 5 \rho_7^2 + 8 b \rho_7^3 - 5 \rho_7^4 + \rho_7^6}{(1 - 2 b \rho_7 + \rho_7^2) (1 - \rho_7^4)}\right)^2-1\right|\\
&=2\left|\frac{1 - 5 \rho_7^2 + 8 b \rho_7^3 - 5 \rho_7^4 + \rho_7^6}{(1 - 2 b \rho_7 + \rho_7^2) (1 - \rho_7^4)}\right|
=2\left|\frac{z(f_{b})'(z)}{f_{b}(z)}\right|.
\end{align*}
This shows that the number $\rho_7$ is the $S^{*}_{\leftmoon}$ radius for the class $\mathcal{K}^{1}_{b}$.

\item The number $\rho_8$ is the smallest real root of the equation in (\ref{K3-S-R}).
Observe that for  $r \leq ((\sqrt{2}-1)/(\sqrt{2}+1))^{1/4} \approx 0.643594$, we have $2\sqrt{2}-2 < a \leq \sqrt{2}$.
Then, by \cite[Lemma 2.2]{MR3496681}, the function $f \in S^{*}_{R}$ if
\[\frac{\tilde{n} r+6 r^2+4 \tilde{n} r^3+6 r^4+\tilde{n} r^5}{(1 + \tilde{n} r + r^2) (1 - r^4)} \leq 2- 2 \sqrt{2} +\frac{1+r^4}{1-r^4}.\]
Thus, the disc (\ref{discforK3}) is contained in the region $\phi_{0}(\mathbb{D})$, where $\phi_{0}(z):=1+(z/k)((k+z)/(k-z)), k=1+\sqrt{2}$ whenever $0 < r \leq \rho_{8}$.
For the function $f_{b}$ in (\ref{function-k3}), we have at $z=-i \rho_8$,
\begin{align*}
\frac{z(f_{b})'(z)}{f_{b}(z)}
&=\frac{1 - 5 \rho_8^2 + 8 b \rho_8^3 - 5 \rho_8^4 + \rho_8^6}{(1 - 2 b \rho_8 + \rho_8^2) (1 - \rho_8^4)}\\
&=2\sqrt{2}-2 = \phi_0(-1).
\end{align*}
Thus, the radius obtained is sharp.
	

\item Let $\rho_{9}$ denote the smallest real root of the equation (\ref{K3-S-Ne}). For $1 \leq a < 5/3$, using \cite[Lemma 2.2]{MR4190740} gives that $f \in S^{*}_{Ne}$ if
\[\frac{\tilde{n} r+6 r^2+4 \tilde{n} r^3+6 r^4+\tilde{n} r^5}{(1 + \tilde{n} r + r^2) (1 - r^4)} \leq \frac{5}{3}-\frac{1+r^4}{1-r^4}.\]
Hence, the disc (\ref{discforK3}) lies in the region $\phi_{Ne}(\mathbb{D})$, where $\phi_{Ne}(z)=1+z-z^3/3$ provided $0 < r \leq \rho_{9}$.
For sharpness, consider $f_{b}$ in (\ref{K3-SL-sharp}) at $z=\rho_{9}$,
\begin{align*}
\left|\frac{z(f_{b})'(z)}{f_{b}(z)}\right|
&=\left|\frac{1-4 b \rho_{9}+6 \rho_{9}^2-4 b \rho_{9}^3+ \rho_{9}^4}{\left(1-2 b \rho_{9}+\rho_{9}^2\right) \left(1-\rho_{9}^2\right)}\right|\\
&=\frac{5}{3}= \phi_{Ne}(1).
\end{align*}

\item Let $\rho_{10}$ denote the smallest real root of the equation (\ref{K3-S-SG}) and suppose $1 \leq a < 2e/(1+e)$. Using \cite[Lemma 2.2]{MR4044913}, it follows that $f \in S^{*}_{SG}$ if
\[\frac{\tilde{n} r+6 r^2+4 \tilde{n} r^3+6 r^4+\tilde{n} r^5}{(1 + \tilde{n} r + r^2) (1 - r^4)} \leq \frac{2e}{1+e}-\frac{1+r^4}{1-r^4}.\]
Hence, the disc (\ref{discforK3}) lies in the region $\phi_{SG}(\mathbb{D})$, where $\phi_{SG}(z)=2/(1+e^{-z})$ provided $0 < r \leq \rho_{10}$.
Further, if $w=z(f_{b})'(z)/(f_{b})(z)$ for $f_{b}$ in (\ref{K3-SL-sharp}), then at $z=\rho_{10}$, we have,
\begin{align*}
\left|\log\left(\frac{w}{2-w}\right)\right|
&=\left|\log\left(\frac{1-4 b \rho_{10} +6 \rho_{10}^2-4 b \rho_{10}^3+ \rho_{10}^4}{1-6 \rho_{10} ^2+8 b \rho_{10}^3-3 \rho_{10}^4}\right)\right|\\
&=\left|\log\left(e\right)\right|
=1
\end{align*}
\end{enumerate}
and hence the estimate is sharp.
\end{proof}

The following theorem provides various starlikeness for the class $\mathcal{K}^{2}_{b,c}$.

\begin{theorem}\label{theorem 2}
If $m=|4b-2c| \leq 2$ and $n=|2c|$, then for class $\mathcal{K}^{2}_{b,c}$, the following statements hold:
\begin{enumerate}[(i)]
\item The $S^{*}_{P}$ radius is the smallest positive real root of the equation
\begin{align}
&1-(m+n) r-3(6+ m n) r^2-17( m+ n) r^3-12(3+ m n) r^4-15( m+ n)r^5 \nonumber \\
&-(14+m n) r^6+(m+n) r^7+3 r^8=0. \label{K1-S-P}
\end{align}

\item For any $\alpha \in [0,1)$, the $S^{*}(\alpha)$ radius is the smallest positive real root of the equation
\begin{align}
&1-\alpha - \alpha(m+n)r - (mn+\alpha mn +8 +2 \alpha)r^2 - (m + n)(8 + \alpha)r^3 - (6 m n + 18) r^4 \nonumber \\
&- (8 - \alpha)(m+n)r^5 - (mn-\alpha mn+8-2\alpha)r^6 + \alpha(m + n) r^7 + (1 + \alpha)r^8=0. \label{K1-S-alpha}
\end{align}

\item The $S^{*}_{L}$ radius is the smallest positive real root of the equation
\begin{align}
&1-\sqrt{2}+\left(2-\sqrt{2}\right) (m+n) r+(11-\sqrt{2}+3 m n-\sqrt{2} m n) r^2+(8 m+8 n)r^3 \nonumber \\
&+(11+\sqrt{2}+3 m n+\sqrt{2} m n) r^4+\left(2+\sqrt{2}\right) (m+n) r^5+(1+\sqrt{2}) r^6=0. \label{K1-S-L}
\end{align}

\item The $S^{*}_{e}$ radius is the smallest positive real root of the equation
\begin{align}
&1-e+(m+n) r+(2+8 e+m n+e m n) r^2+(m+8 e m+n+8 e n) r^3+(18 e+6 e m n) r^4 \nonumber \\
&-(m-8 e m+n-8 e n) r^5-(2-8 e+m n-e m n) r^6-(m+n) r^7-(1+e) r^8 =0. \label{K1-S-e}
\end{align}

\item The $S^{*}_{c}$ radius is the smallest positive real root of the equation
\begin{align}
&2-(m+n) r-(26+4 m n) r^2-(25 m+25 n) r^3-(54+18 m n) r^4-(23 m+23 n) r^5\nonumber \\
&-(22+2 m n) r^6+(m+n) r^7+4 r^8=0. \label{K1-S-c}
\end{align}

\item The $S^{*}_{\sin}$ radius is the smallest positive real root of the equation
\begin{align}
&\sin 1 - (m + n - m \sin 1 - n \sin 1)r - (10 + 2 m n - 2 \sin 1 - m n \sin 1)r^2 - (9 m + 9 n \nonumber \\
&- m \sin 1 - n \sin 1)r^3 -(22 + 6 m n) r^4 -(11 m + 11 n + m \sin 1 + n \sin 1)r^5 -(14 + 4 m n \nonumber \\
&+ 2 \sin 1 + m n \sin 1)r^6 - (3 m + 3 n + m \sin 1 + n \sin 1)r^7 - (2 + \sin 1)r^8 =0. \label{K1-S-sin}
\end{align}

\item The $S^{*}_{\leftmoon}$ radius is the smallest positive real root of the equation
\begin{align}
&(-2+\sqrt{2})+(-3+\sqrt{2}) (m+n) r+(2 (-7+\sqrt{2})+(-4+\sqrt{2}) m n) r^2+(-11+\sqrt{2}) \nonumber\\
&(m+n) r^3+(-22-6 m n) r^4-(9+\sqrt{2}) (m+n) r^5+(-2 (5+\sqrt{2})-(2+\sqrt{2}) m n) r^6 \nonumber\\
&-(1+\sqrt{2}) (m+n) r^7-\sqrt{2} r^8=0. \label{K1-S-lune}
\end{align}

\item The $S^{*}_{R}$ radius is the smallest positive real root of the equation
\begin{align}
&3-2 \sqrt{2}+(2 m-2 \sqrt{2} m+2 n-2 \sqrt{2} n) r-(4+4 \sqrt{2}-m n+2 \sqrt{2} m n) r^2-(6 m+2 \sqrt{2} m \nonumber \\
&+6 n +2 \sqrt{2} n) r^3-(18+6 m n) r^4-(10 m-2 \sqrt{2} m+10 n-2 \sqrt{2} n) r^5 -(12-4 \sqrt{2}+3 m n \nonumber \\
&-2 \sqrt{2} m n) r^6-(2 m-2 \sqrt{2} m+2 n-2 \sqrt{2} n) r^7-(1-2 \sqrt{2}) r^8 =0. \label{K1-S-R}
\end{align}	


\item The $S^{*}_{N_{e}}$ radius is the smallest positive real root of the equation
\begin{align}
&2-(m+n) r-(26+4 m n) r^2-(25 m+25 n) r^3-(66+18 m n) r^4-(35 m+35 n) r^5 \nonumber \\
&-(46+14 m n) r^6-(11 m+11 n) r^7-8 r^8=0. \label{K1-S-Ne}
\end{align}

\item The $S^{*}_{SG}$ radius is the smallest positive real root of the equation
\begin{align}
&1-e+(2 m+2 n) r+(12+8 e+3 m n+e m n) r^2+(10 m+8 e m+10 n+8 e n) r^3+(22 \nonumber \\
&+22 e+6 m n+6 e m n) r^4 +(10 m+12 e m+10 n+12 e n) r^5 +(12+16 e+3 m n+5 e m n) r^6 \nonumber \\
&+(2 m+4 e m+2 n+4 e n) r^7+(1+3 e) r^8 =0. \label{K1-S-SG}
\end{align}
\end{enumerate}
All estimates are sharp.
\end{theorem}

\begin{proof}
Let $f \in \mathcal{K}^{2}_{b,c}$. Choose a function $g \in \mathcal{A}$ such that for all $z \in \mathbb{D}$, $\operatorname{Re}(f(z)/g(z))>0$ and $\operatorname{Re}((g(z)(1-z^2))/z)>0$.
Observe that the complex valued functions $p_1,p_2$ defined on the unit disc $\mathbb{D}$ by $p_1(z)=f(z)/g(z)$
and
$p_2(z)=g(z)(1-z^2)/z$
are functions in $\mathcal{P}(0)$.
Moreover, we have
\[f(z)=\frac{z p_1(z)p_2(z)}{1-z^2},\]
which satisfy the relation
\begin{equation}\label{logdiff}
\frac{zf'(z)}{f(z)}=\frac{zp_1'(z)}{p_1(z)}+\frac{zp_2'(z)}{p_2(z)}+\frac{1+z^2}{1-z^2}.
\end{equation}
Substituting $\alpha=0$ in (\ref{mainlemma}) and using (\ref{transformmap}) and (\ref{logdiff}), we get
\begin{equation}\label{discforK1}
\left|\frac{zf'(z)}{f(z)}-\frac{1+r^4}{1-r^4}\right| \leq
\frac{\splitfrac{(m+n) r+(10+2 m n) r^2+(9 m+9 n) r^3+(20+6 m n) r^4}{+(9 m+9 n) r^5+(10+2 m n) r^6+(m+n) r^7}}{(1 + m r +r^2) (1 + n r + r^2) (1 - r^4)}.
\end{equation}

Also, a straightforward calculation shows that
\begin{align}
\operatorname{Re}\left(\frac{zf'(z)}{f(z)}\right) & \geq \frac{1+r^4}{1-r^4}-\frac{\splitfrac{(m+n) r+(10+2 m n) r^2+(9 m+9 n) r^3+(20+6 m n) r^4}{+(9 m+9 n) r^5+(10+2 m n) r^6+(m+n) r^7}}{(1 + m r +r^2) (1 + n r + r^2) (1 - r^4)} \nonumber \\
&=\frac{\splitfrac{r^8 - (8 + m n) r^6 - 8 (m + n) r^5 - 6 (3 + m n) r^4 - 8 (m + n) r^3 }{-(8 + mn) r^2 + 1}}{(1 + m r + r^2) (1 + n r +
r^2) (1 - r^4)}. \label{K2-realpart}
\end{align}

\begin{enumerate}[(i)]
\item Set $x(r):=1-(m+n)r-3(6+ m n)r^2-17( m+ n)r^3-12(3+ m n)r^4-15( m+ n)r^5 -(14+m n)r^6+(m+n)r^7+3r^8$.
Note that $x(0)=1 > 0$ and $x(1)=-16(mn+2m+2n+4) <0$ and thus in view of the intermediate value theorem, a root of the equation (\ref{K1-S-P}) lies in the interval $(0,1)$, denoted by
$\rho_1$.

When $1/2 < a \leq 3/2$, in view of \cite[Lemma 2.2]{MR1415180}, the disc (\ref{discforK1}) is contained in the region $\{w \in \mathbb{C}: |w-1| < \operatorname{Re}w\}$ if
\[\frac{\splitfrac{(m+n) r+(10+2 m n) r^2+(9 m+9 n) r^3+(20+6 m n) r^4}{+(9 m+9 n) r^5+(10+2 m n) r^6+(m+n) r^7}}{(1 + m r +r^2) (1 + n r + r^2) (1 - r^4)} \leq \frac{1+r^4}{1-r^4}-\frac{1}{2}.\]
Hence,
\[\operatorname{Re}\left(\frac{zf'(z)}{f(z)}\right) > \left|\frac{zf'(z)}{f(z)}-1\right|\]
whenever $0<r\leq \rho_1$.
This proves that the $S^{*}_P$ radius for the class $\mathcal{K}^{2}_{b,c}$ is the number $\rho_1$.

To justify the sharpness of $S^{*}_P$ radius, observe that the function $f_{b,c}$ in (\ref{f1}) at $z=-i \rho_1$ satisfies
\begin{align*}
\operatorname{Re}\left(\frac{z(f_{b,c})'(z)}{f_{b,c}(z)}\right)
&=\frac{\splitfrac{\rho_1^8-(8-2 c (-4 b+2 c)) \rho_1^6+ 32 b \rho_1^5 -6 (3-2 c (-4 b+2 c)) \rho_1^4 +32 b \rho_1^3}{-(8-2 c (-4 b+2 c)) \rho_1^2+1}}{\left(1-2 c \rho_1+\rho_1^2\right) \left(1-(4 b-2 c) \rho_1+\rho_1^2\right) \left(1-\rho_1^4\right)}\\
&=\left|\frac{\splitfrac{4 b \rho_1-\left(10+16 b c+8 c^2\right) \rho_1^2+36  b \rho_1^3-\left(18+48 b c-24 c^2\right) \rho_1^4}{+28  b \rho_1^5-6 \rho_1^6-4 b \rho_1^7+2 \rho_1^8}}{(1-2 c \rho_1+\rho_1^2) (1-4 b \rho_1+2 c \rho_1+\rho_1^2) (1-\rho_1^4)}\right|\\
&=\left|\frac{z(f_{b,c})'(z)}{f_{b,c}(z)}-1\right|.
\end{align*}

\begin{remark}
On the similar lines, we can prove the sharpness for $b>0$. The sharpness is attained at the point
\[z = \left\{
\begin{array}{ll}
		-i \rho, & \text{if } b<0\\
		i \rho,  & \text{if } b>0\\
\end{array}
\right.\]
for the classes $S^{*}_{P}$, $S^{*}(\alpha)$, $S^{*}_{e}$, $S^{*}_{c}$, $S^{*}_{\leftmoon}$, $S^{*}_{R}$, $S^{*}_{Ne}$, $S^{*}_{SG}$ and for the classes $S^{*}_{L}$ and $S^{*}_{\sin}$, sharpness is attained at the point
	\[z = \left\{
	\begin{array}{ll}
		\rho, & \text{if } b<0\\
		- \rho, & \text{if } b>0.\\
	\end{array}
	\right.\]
\end{remark}

\begin{figure}[h!]
	\centering
	\begin{subfigure}[!]{0.25\linewidth}
		\includegraphics[width=\linewidth]{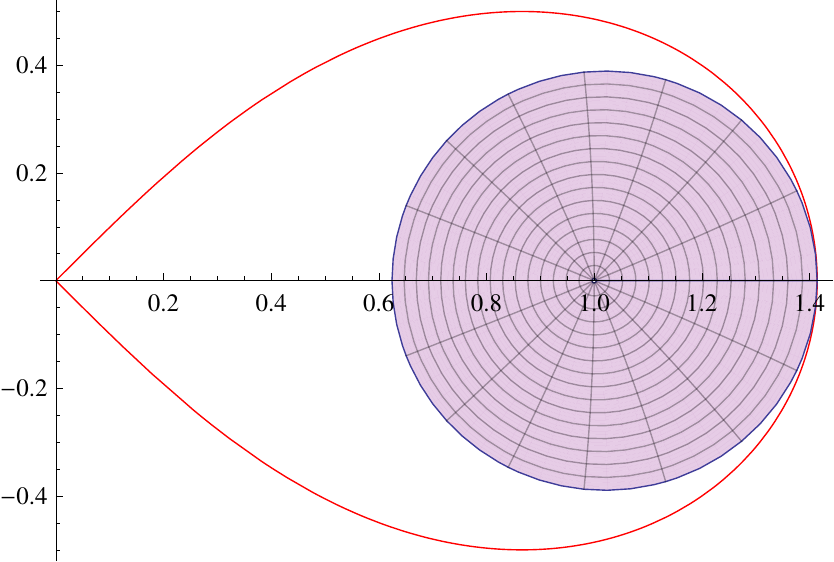}
		\caption{Sharpness of class $S^{*}_{L}$ with $\rho_3=0.116675$ at $b=-1,c=-1$}
	\end{subfigure}
	\hspace{2.6em}
	\begin{subfigure}[!]{0.23\linewidth}
		\includegraphics[width=\linewidth]{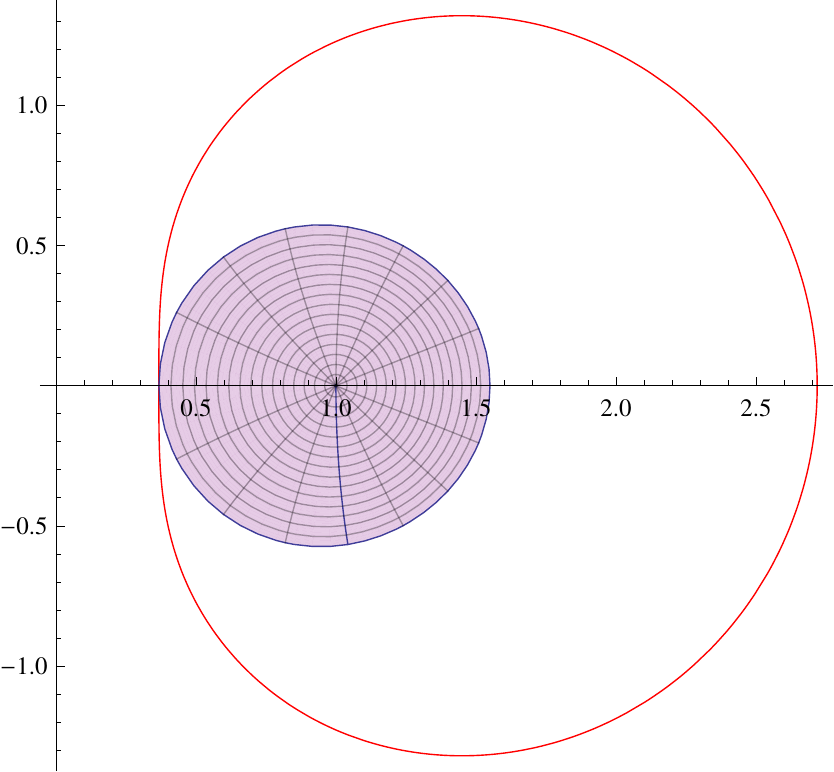}
		\caption{Sharpness of class $S_{e}^{*}$ with $\rho_4=0.144684$ at $b=-1,c=-1$}
	\end{subfigure}
	\hspace{2.6em}
	\begin{subfigure}[!]{0.22\linewidth}
		\includegraphics[width=\linewidth]{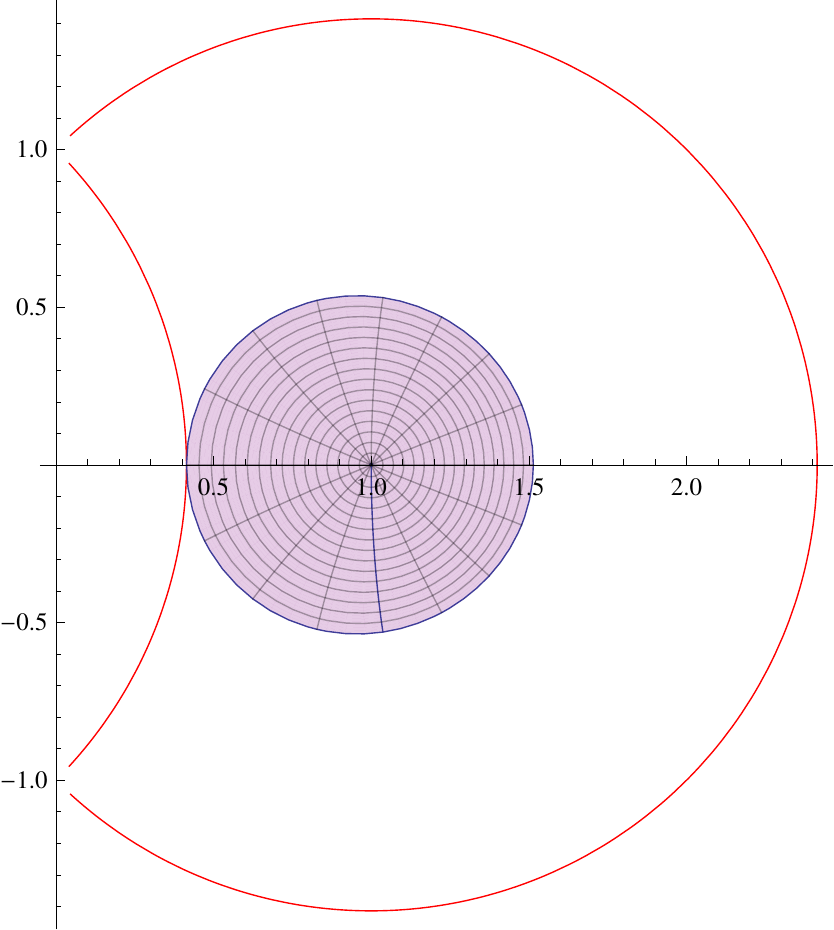}
		\caption{Sharpness of class $S^{*}_{\leftmoon}$ with $\rho_7=0.134993$ at  $b=-1,c=-1$}
	\end{subfigure}
	\caption{Graphical illustration of sharpness for classes in Theorem (\ref{theorem 2}) for particular choices of $b$ and $c$.}
	\label{fig:coffee}
\end{figure}

\item Let $f \in \mathcal{K}^{2}_{b,c}$ and $\alpha \in [0,1)$.
Let $\rho_2$ denote the smallest positive real root of the equation (\ref{K1-S-alpha}) in $(0,1)$.
From (\ref{K2-realpart}), it follows that for $0<r \leq \rho_2$, the function $f$ satisfies
\begin{align*}
\operatorname{Re}\left(\frac{zf'(z)}{f(z)}\right) &=\frac{\splitfrac{r^8 - (8 + m n) r^6 - 8 (m + n) r^5 - 6 (3 + m n) r^4 - 8 (m + n) r^3 }{-(8 + mn) r^2 + 1}}{(1 + m r + r^2) (1 + n r +
r^2) (1 - r^4)}\\
&> \alpha,
\end{align*}
thereby showing that $f \in S^{*}(\alpha)$ in each disc for $0<r \leq \rho_2$.
For $z=-i \rho_2$, the function $f_{b,c}$ defined in (\ref{f1}) satisfies
\begin{align*}
\frac{z(f_{b,c})'(z)}{f_{b,c}(z)}
&=\frac{\splitfrac{\rho_2^8-(8-2 c (-4 b+2 c)) \rho_2^6+ 32b \rho_2^5 -6 (3-2 c (-4 b+2 c)) \rho_2^4 + 32 b \rho_2^3}{-(8-2 c (-4 b+2 c)) \rho_2^2+1}}{\left(1-2 c \rho_2+\rho_2^2\right) \left(1-(4b-2 c) \rho_2+\rho_2^2\right) \left(1-\rho_2^4\right)}\\
&=\alpha
\end{align*}
proving that the radius is sharp.

\item The number $\rho_3$ is the smallest real root of the equation in (\ref{K1-S-L}).
From (\ref{discforK1}), it follows that
\[\left|\frac{zf'(z)}{f(z)}-1\right| \leq \frac{(m+n) r+2(5+m n) r^2+8(m+n) r^3+4(3+m n) r^4+3(m+n) r^5+2 r^6}{(1 + m r +
r^2) (1 + n r + r^2)(1 - r^2)}.\]
When $2\sqrt{2}/3 \leq a < \sqrt{2}$, by \cite[Lemma 2.2]{MR2879136}, the disc (\ref{discforK1}) is contained in the lemniscate region $\{w \in \mathbb{C}: |w^2-1| < 1\}$ if
\[\frac{(m+n) r+2(5+m n) r^2+8(m+n) r^3+4(3+m n) r^4+3(m+n) r^5+2 r^6}{(1 + m r +r^2) (1 + n r + r^2)(1 - r^2)} \leq \sqrt{2}-1.\]
Hence,
\[\left|\left(\frac{zf'(z)}{f(z)}\right)^2-1\right| < 1\]
for $0 < r \leq \rho_3$ showing that the number $\rho_3$ is the $S^{*}_{L}$ radius for the class $\mathcal{K}^{2}_{b,c}$.
Consider the functions $f_{b,c}$, $g_{b,c}: \mathbb{D} \to \mathbb{C}$ defined by
\begin{equation}\label{K1-SL-sharp}
f_{b,c}=\frac{z(1-2cz+z^2)(1-(4 b-2 c)z+z^2)}{(1-z^2)^3} \quad \text{and} \quad
g_{b,c}=\frac{z \left(1-2 c z+z^2\right)}{\left(1-z^2\right)^2}.
\end{equation}
The function $f_{b,c}$ for the chosen $g_{b,c}$ as given in (\ref{K1-SL-sharp}) is a member of $\mathcal{K}^{2}_{b,c}$ because it satisfy
\[\frac{f_{b,c}(z)}{g_{b,c}(z)}=\frac{1-s_3(z)}{1+s_3(z)} \quad \text{and} \quad
\frac{g_{b,c}(z)(1-z^2)}{z}=\frac{1-s_4(z)}{1+s_4(z)},\]
where
\[s_3(z)=\frac{z(z-(2b-c))}{(2b-c)z-1} \quad \text{and} \quad s_4(z)=\frac{z (z-c)}{cz-1}\]
are analytic functions satisfying the conditions of Schwarz's lemma in unit disc and hence
$\operatorname{Re}(f_{b,c}(z)/g_{b,c}(z)) > 0$ and $\operatorname{Re}(g_{b,c}(z)(1-z^2)/z)>0$.
Then for $f_{b,c}$ in (\ref{K1-SL-sharp}) at $z=\rho_{3}$, we have,
\begin{align*}
\left|\left(\frac{z(f_{b,c})'(z)}{f_{b,c}(z)}\right)^2-1\right|
&=\left|\left(\frac{\splitfrac{1-8 b \rho_{3}+11 \rho_{3}^2+24 b c \rho_{3}^2-12 c^2 \rho_{3}^2-32 b \rho_{3}^3+11 \rho_{3}^4}{+24 b c \rho_{3}^4-12 c^2 \rho_{3}^4-8 b \rho_{3}^5+\rho_{3}^6}}{(1-2 c \rho_{3}+\rho_{3}^2) (1-(4 b -2 c)\rho_{3}+\rho_{3}^2)(1-\rho_{3}^2)}\right)^2-1\right|\\
&=|(\sqrt{2})^2-1|
=1
\end{align*}
and therefore, the estimate is sharp.

\item  The number $\rho_4$ is the smallest real root of the equation in (\ref{K1-S-e}).
When $1/e < a \leq (e+1/e)/2$, in view of \cite[Lemma 2.2]{MR3394060}, the disc (\ref{discforK1}) is contained in the region $\{w \in \mathbb{C}: |\log w| < 1\}$ if
\[\frac{\splitfrac{(m+n) r+(10+2 m n) r^2+(9 m+9 n) r^3+(20+6 m n) r^4}{+(9 m+9 n) r^5+(10+2 m n) r^6+(m+n) r^7}}{(1 + m r +r^2) (1 + n r + r^2) (1 - r^4)} \leq \frac{1+r^4}{1-r^4}-\frac{1}{e}\]
for $0<r\leq \rho_4$.
This shows that the $S^{*}_{e}$ radius for the class $\mathcal{K}^{2}_{b,c}$ is the number $\rho_4$.
Moreover, for $f_{b,c}$ in (\ref{f1}) at $z=-i \rho_4$, we have,
\begin{align*}
\left|\log\left(\frac{z(f_{b,c})'(z)}{f_{b,c}(z)}\right)\right|
&= \left|\log\left(\frac{\splitfrac{\rho_4^8-(8-2 c (-4 b+2 c)) \rho_4^6+ 32b \rho_4^5 -6 (3-2 c (-4 b}{+2 c)) \rho_4^4 +32 b \rho_4^3-(8-2 c (-4 b+2 c)) \rho_4^2+1}}{\left(1-2 c \rho_4+\rho_4^2\right) \left(1-(4 b-2 c) \rho_4+\rho_4^2\right) \left(1-\rho_4^4\right)}\right)\right|\\
&=\left|\log\left(\frac{1}{e}\right)\right|
=1.
\end{align*}
This shows that the radius estimate is sharp.

\item The number $\rho_5$ is the smallest real root of the equation in (\ref{K1-S-c}). In view of \cite[Lemma 2.5]{MR3536076}, the disc (\ref{discforK1}) is contained in the region $\phi_{c}(\mathbb{D})$, where $\phi_{c}(z)=1+(4/3)z+(2/3)z^2$ if
\[\frac{\splitfrac{(m+n) r+(10+2 m n) r^2+(9 m+9 n) r^3+(20+6 m n) r^4}{+(9 m+9 n) r^5+(10+2 m n) r^6+(m+n) r^7}}{(1 + m r +r^2) (1 + n r + r^2) (1 - r^4)} \leq \frac{1+r^4}{1-r^4}-\frac{1}{3}\]
whenever $1/3 < a \leq 5/3$.
The result is sharp for the function $f_{b,c}$ given in (\ref{f1}) and at $z=-i \rho_5$, we have,
\begin{align*}
\left|\frac{z(f_{b,c})'(z)}{f_{b,c}(z)}\right|
&=\left|\frac{\splitfrac{\rho_5^8-(8-2 c (-4 b+2 c)) \rho_5^6+ 32b \rho_5^5 -6 (3-2 c (-4 b + 2 c)) \rho_5^4}{+32 b \rho_5^3-(8-2 c (-4 b+2 c)) \rho_5^2+1}}{\left(1-2 c \rho_5+\rho_5^2\right) \left(1-(4 b-2 c) \rho_5+\rho_5^2\right) \left(1-\rho_5^4\right)}\right|\\
&=\frac{1}{3} = \phi_c(-1).
\end{align*}

\item The number $\rho_6$ is the smallest real root of the equation in (\ref{K1-S-sin}).
When $1-\sin 1 < a \leq 1+\sin 1$, an application of \cite[Lemma 3.3]{MR3913990} yields that the function $f \in S^{*}_{\sin}$ if
\[\frac{\splitfrac{(m+n) r+(10+2 m n) r^2+(9 m+9 n) r^3+(20+6 m n) r^4}{+(9 m+9 n) r^5+(10+2 m n) r^6+(m+n) r^7}}{(1 + m r +r^2) (1 + n r + r^2) (1 - r^4)} \leq \sin 1 -\frac{2r^4}{1-r^4}.\]
For $0 < r \leq \rho_6$, the disc (\ref{discforK1}) is contained in the region $\phi_{s}(\mathbb{D})$, where $\phi_{s}(z)=1+\sin z$, showing that the radius of sine starlikeness for the class $\mathcal{K}^{2}_{b,c}$ is the number $\rho_6$.
To prove sharpness, observe that for the functions given in (\ref{K1-SL-sharp}), at $z=\rho_6$, we have,
\begin{align*}
\frac{z(f_{b,c})'(z)}{f_{b,c}(z)}
&=\frac{\splitfrac{\rho_6^8-(8-2 c (-4 b+2 c)) \rho_6^6+ 32b \rho_6^5 -6 (3-2 c (-4 b + 2 c)) \rho_6^4}{+32 b \rho_6^3-(8-2 c (-4 b+2 c)) \rho_6^2+1}}{\left(1-2 c \rho_6+\rho_6^2\right) \left(1-(4 b-2 c) \rho_6+\rho_6^2\right) \left(1-\rho_6^4\right)}\\
&=1+\sin 1 = \phi_s(1) \in \partial\phi_s(\mathbb{D}).
\end{align*}
Hence, the result is sharp.

\vskip 0.3cm

\item The number $\rho_7$ is the smallest real root of the equation in (\ref{K1-S-lune}).
Consider $\sqrt{2}-1 < a < \sqrt{2}+1$ and applying \cite[Lemma 2.1]{MR3718233}, the disc (\ref{discforK1}) is contained in the region $\{w \in \mathbb{C}: 2|w| > |w^2-1|\}$ if
\[\frac{\splitfrac{(m+n) r+(10+2 m n) r^2+(9 m+9 n) r^3+(20+6 m n) r^4}{+(9 m+9 n) r^5+(10+2 m n) r^6+(m+n) r^7}}{(1 + m r +r^2) (1 + n r + r^2) (1 - r^4)} \leq \frac{1+r^4}{1-r^4}+1-\sqrt{2}.\]
Therefore, for $0<r \leq \rho_7$, we have
\[2\left|\frac{zf'(z)}{f(z)}\right| > \left|\left(\frac{zf'(z)}{f(z)}\right)^2-1\right|\]
which concludes that the number $\rho_7$ is the $S^{*}_{\leftmoon}$ radius for the class $\mathcal{K}^{2}_{b,c}$.
Further, for the function in (\ref{f1}) at $z=-i \rho_7$, we have,
\begin{align*}
\left|\left(\frac{z(f_{b,c})'(z)}{f_{b,c}(z)}\right)^2-1\right|
&=\left|\left(\frac{\splitfrac{\rho_7^8-(8-2 c (-4 b+2 c)) \rho_7^6+ 32b \rho_7^5 -6 (3-2 c (-4 b}{+2 c)) \rho_7^4 + 32 b \rho_7^3-(8-2 c (-4 b+2 c)) \rho_7^2+1}}{\left(1-2 c \rho_7+\rho_7^2\right) \left(1-(4b-2 c) \rho_7+\rho_7^2\right) \left(1-\rho_7^4\right)}\right)^2-1\right|\\
&=2\left|\frac{\splitfrac{\rho_7^8-(8-2 c (-4 b+2 c)) \rho_7^6+ 32b \rho_7^5 -6 (3-2 c (-4 b+2 c)) \rho_7^4}{+ 32 b \rho_7^3-(8-2 c (-4 b+2 c)) \rho_7^2+1}}{\left(1-2 c \rho_7+\rho_7^2\right) \left(1-(4b-2 c) \rho_7+\rho_7^2\right)\left(1-\rho_7^4\right)}\right|\\
&=2\left|\frac{z(f_{b,c})'(z)}{f_{b,c}(z)}\right|.
\end{align*}
Thus, the estimate is best possible.

\vskip 0.3cm

\item The number $\rho_8$ is the smallest real root of the equation in (\ref{K1-S-R}).
If $2(\sqrt{2}-1) < a \leq \sqrt{2}$, by \cite[Lemma 2.2]{MR3496681} the function $f \in S^{*}_{R}$ if
\[\frac{\splitfrac{(m+n) r+(10+2 m n) r^2+(9 m+9 n) r^3+(20+6 m n) r^4}{+(9 m+9 n) r^5+(10+2 m n) r^6+(m+n) r^7}}{(1 + m r +r^2) (1 + n r + r^2) (1 - r^4)} \leq \frac{1+r^4}{1-r^4} - 2(\sqrt{2}-1).\]
Observe that for $0 < r \leq \rho_8$, the disc (\ref{discforK1}) is contained in the region $\phi_{0}(\mathbb{D})$, where $\phi_{0}(z):=1+(z/k)((k+z)/(k-z)), k=1+\sqrt{2}$.
For sharpness, note that the function $f_{b,c}$ in (\ref{f1}) at $z=-i \rho_8$ satisfies
\begin{align*}
\frac{z(f_{b,c})'(z)}{f_{b,c}(z)}
&=\frac{\splitfrac{\rho_8^8-(8-2 c (-4 b+2 c)) \rho_8^6+ 32b \rho_8^5 -6 (3-2 c (-4 b+2 c)) \rho_8^4+ 32 b \rho_8^3}{-(8-2 c (-4 b+2 c)) \rho_8^2+1}}{\left(1-2 c \rho_8+\rho_8^2\right) \left(1-(4b-2 c) \rho_8+\rho_8^2\right)\left(1-\rho_8^4\right)}\\
&=2\sqrt{2}-2 = \phi_0(-1).
\end{align*}


\item Let $\rho_{9}$ denote the smallest real root of the equation (\ref{K1-S-Ne}).
For $1 \leq a < 5/3$, an application of \cite[Lemma 2.2]{MR4190740} gives that $f \in S^{*}_{Ne}$ if
\[\frac{\splitfrac{(m+n) r+(10+2 m n) r^2+(9 m+9 n) r^3+(20+6 m n) r^4}{+(9 m+9 n) r^5+(10+2 m n) r^6+(m+n) r^7}}{(1 + m r +r^2) (1 + n r + r^2) (1 - r^4)} \leq \frac{5}{3}-\frac{1+r^4}{1-r^4}.\]
Thus, for $0 < r \leq \rho_{9}$, the disc (\ref{discforK1}) lies in the region $\phi_{Ne}(\mathbb{D})$, where $\phi_{Ne}(z)=1+z-z^3/3$.
Further, the function defined in (\ref{f1}) at $z=-i \rho_{9}$ satisfies
\begin{align*}
\left|\frac{z(f_{b,c})'(z)}{f_{b,c}(z)}\right|
&=\left|\frac{\splitfrac{\rho_{9}^8-(8-2 c (-4 b+2 c)) \rho_{9}^6+ 32b \rho_{9}^5 -6 (3-2 c (-4 b+2 c)) \rho_{9}^4}{+ 32 b \rho_{9}^3-(8-2 c (-4 b+2 c)) \rho_{9}^2+1}}{\left(1-2 c \rho_{9}+\rho_{9}^2\right) \left(1-(4b-2 c)\rho_{9}+\rho_{9}^2\right)\left(1-\rho_{9}^4\right)}\right|\\
&=\frac{5}{3} = \phi_{Ne}(1).
\end{align*}
This proves that the estimate is best possible.

\item Let $f \in \mathcal{K}^{2}_{b,c}$ and $\rho_{10}$ denote the smallest real root of the equation (\ref{K1-S-SG}).
Using \cite[Lemma 2.2]{MR4044913} for $1 \leq a < 2e/(1+e)$, it follows that $f \in S^{*}_{SG}$ if
\[\frac{\splitfrac{(m+n) r+(10+2 m n) r^2+(9 m+9 n) r^3+(20+6 m n) r^4}{+(9 m+9 n) r^5+(10+2 m n) r^6+(m+n) r^7}}{(1 + m r +r^2) (1 + n r + r^2) (1 - r^4)} \leq \frac{2e}{1+e}-\frac{1+r^4}{1-r^4}.\]
Hence, for $0 < r \leq \rho_{10}$, the disc (\ref{discforK1}) lies in the region $\phi_{SG}(\mathbb{D})$, where $\phi_{SG}(z)=2/(1+e^{-z})$.
This shows that the $S^{*}_{SG}$ radius for the class $\mathcal{K}^{2}_{b,c}$ is the number $\rho_{10}$.
For $w=z(f_{b,c})'/f_{b,c}$ and at $z=-i \rho_{10}$, where $f_{b,c}$ is the function defined in (\ref{f1}), we have
\begin{align*}
\left|\log\left(\frac{w}{2-w}\right)\right|
&=\left|\frac{\splitfrac{-1+(8+8 b c-4 c^2) \rho_{10}^2+32 b \rho_{10}^3+6 (3+8 b c-4 c^2) \rho_{10}^4}{+32 b \rho_{10}^5+(8+8 b c-4 c^2) \rho_{10}^6-\rho_{10}^8}}{\splitfrac{-1+8 b \rho_{10} -(12 +24 b c -12 c^2) \rho_{10} ^2+40 b \rho_{10} ^3-(18+48 b c}{-24 c^2 )\rho_{10} ^4+24 b \rho_{10} ^5-(4 -8 b c +4 c^2) \rho_{10} ^6-8 b \rho_{10} ^7+3 \rho_{10} ^8}}\right|\\
&=1.
\end{align*}
\end{enumerate}
Thus, the obtained radius is sharp.
\end{proof}

In the theorem given below, we determine the radii constants for the class $\mathcal{K}^{3}_{b,c}$.

\begin{theorem}\label{theorem 3}
Let $u=|2c-3b| \leq 1$ and $v=|2c|$. Then, the following results hold for the class $\mathcal{K}^{3}_{b,c}$:
\begin{enumerate}[(i)]
\item The $S^{*}_{P}$ radius is the smallest positive real root of the equation
\begin{align}
&1-(u+v)r-(15+3 u v)r^2-(17 u+12 v)r^3-(17+12 u v)r^4-(15 u+3 v)r^5 \nonumber \\
&-(1+u v)r^6+ur^7=0. \label{K2-S-P}
\end{align}

\item For any $0 \leq \alpha <1$, the $S^{*}(\alpha)$ radius is the smallest positive real root of the equation
\begin{align}
&1-\alpha-(u \alpha+v \alpha)r-(7+u v+\alpha+u v \alpha )r^2-(8 u+6 v+u \alpha)r^3-(9+6 u v-\alpha )r^4 \nonumber \\
&-(8 u+2 v-u \alpha -v \alpha)r^5 -(1+u v-\alpha -u v \alpha )r^6+ \alpha u r^7=0. \label{K2-S-alpha}
\end{align}

\item The $S^{*}_{L}$ radius is the smallest positive real root of the equation
\begin{align}
&1-\sqrt{2}+(2 u-\sqrt{2} u+2 v-\sqrt{2} v)r +
(8+3 u v-\sqrt{2} u v)r^2+
(8 u+4 v+\sqrt{2} v)r^3 \nonumber \\
&+(3+\sqrt{2}+3 u v+\sqrt{2} u v)r^4+
(2 u+\sqrt{2} u)r^5=0. \label{K2-S-L}
\end{align}

\item The $S^{*}_{e}$ radius is the smallest positive real root of the equation
\begin{align}
&1-e+(u+v)r+(1+7 e+u v+e u v)r^2 +(u+8 e u+6 e v)r^3 -(1-9 e-6 e u v)r^4 \nonumber \\
&-(u-8 e u+v-2 e v)r^5-(1-e+u v-e u v)r^6 -u r^7=0. \label{K2-S-e}
\end{align}

\item The $S^{*}_{c}$ radius is the smallest positive real root of the equation
\begin{align}
&2-(u+v)r-(22+4 u v)r^2-(25 u+18 v)r^3 -(26+18 u v)r^4 -(23 u+5 v)r^5 \nonumber \\
&-(2+2 u v)r^6+u r^7=0. \label{K2-S-c}
\end{align}

\item The $S^{*}_{\sin}$ radius is the smallest positive real root of the equation
\begin{align}
&\sin 1-(u+v-u \sin 1-v \sin 1)r-(8+2 u v-\sin 1-u v \sin 1)r^2-(9 u+6 v-u \sin 1)r^3 \nonumber \\
&-(12+6 u v+\sin 1)r^4 -(11 u+5 v+u \sin 1+v \sin 1)r^5-(4+4 u v+\sin 1+u v \sin 1)r^6 \nonumber \\
&-(3 u+u \sin 1)r^7=0. \label{K2-S-sin}
\end{align}

\item The $S^{*}_{\leftmoon}$ radius is the smallest positive real root of the equation
\begin{align}
&2-\sqrt{2}+(u-\sqrt{2} u+v-\sqrt{2} v)r-(6+\sqrt{2}+\sqrt{2} u v)r^2-(7 u+\sqrt{2} u+6 v)r^3 -(10-\sqrt{2} \nonumber \\
&+6 u v)r^4 -(9 u-\sqrt{2} u+3 v-\sqrt{2} v)r^5-(2-\sqrt{2}+2 u v-\sqrt{2} u v)r^6-(u-\sqrt{2} u)r^7=0. \label{K2-S-lune}
\end{align}

\item The $S^{*}_{R}$ radius is the smallest positive real root of the equation
\begin{align}
&3-2 \sqrt{2}+(2 u-2 \sqrt{2} u+2 v-2 \sqrt{2} v)r-(5+2 \sqrt{2}-u v+2 \sqrt{2} u v)r^2-(6 u+2 \sqrt{2} u+6 v)r^3 \nonumber \\
&-(11-2 \sqrt{2}+6 u v)r^4-(10 u-2 \sqrt{2} u+4 v-2 \sqrt{2} v)r^5-(3-2 \sqrt{2}+3 u v-2 \sqrt{2} u v)r^6 \nonumber \\
&-(2 u-2 \sqrt{2} u)r^7=0. \label{K2-S-R}
\end{align}


\item The $S^{*}_{N_{e}}$ radius is the smallest positive real root of the equation
\begin{align}
&2-(u+v)r-(22+4 u v)r^2-(25 u+18 v)r^3-(38+18 u v)r^4-(35 u+17 v)r^5 \nonumber \\
&-(14+14 u v)r^6-11u r^7=0. \label{K2-S-Ne}
\end{align}

\item The $S^{*}_{SG}$ radius is the smallest positive real root of the equation
\begin{align}
&1-e+(2 u+2 v)r+(9+7 e+3 u v+e u v)r^2+(10 u+8 e u+6 v+6 e v)r^3+(11+13 e+6 u v \nonumber \\
&+6 e u v)r^4 +(10 u+12 e u+4 v+6 e v)r^5+(3+5 e+3 u v+5 e u v)r^6+(2 u+4 e u)r^7 =0.  \label{K2-S-SG}
\end{align}
\end{enumerate}
\end{theorem}

\begin{proof}
Let	$f \in \mathcal{K}^{3}_{b,c}$. Further, choose the function $g \in \mathcal{A}$ such that
\[\left|\frac{f(z)}{g(z)}-1\right| < 1 \text{ and } \operatorname{Re}\frac{g(z)(1-z^2)}{z}>0.\]
Define the functions $p_1, p_2:\mathbb{D} \to \mathbb{C}$ as
$p_1(z)=g(z)(1-z^2)/z$
and
$p_2(z)=g(z)/f(z)$.
Clearly, $p_1 \in \mathcal{P}_{c}$ and $p_2 \in \mathcal{P}_{2c-3b}(1/2)$.
Moreover,
\[f(z)=\frac{g(z)}{p_2(z)}=\frac{zp_1(z)}{(1-z^2)p_2(z)}\]
and
\begin{equation}\label{equationforK2}
\frac{zf'(z)}{f(z)}=\frac{zp_1'(z)}{p_1(z)}-\frac{zp_2'(z)}{p_2(z)}+\frac{1+z^2}{1-z^2}.
\end{equation}
From (\ref{mainlemma}), the following inequalities readily follows:
\begin{equation}\label{lemmainequalityforK2}
\left|\frac{zp_1'(z)}{p_1(z)}\right| \leq \frac{r}{1-r^2}\frac{vr^2+v+4r}{r^2+vr+1}
\quad \text{and} \quad
\left|\frac{zp_2'(z)}{p_2(z)}\right| \leq \frac{r}{1-r^2}\frac{ur^2+u+2r}{ur+1}.
\end{equation}
Thus, combining (\ref{equationforK2}) and (\ref{lemmainequalityforK2}) yields the disk
\begin{equation}\label{discforK2}
\left|\frac{zf'(z)}{f(z)}-\frac{1+r^4}{1-r^4}\right| \leq
\frac{\splitfrac{(u+v)r+(8+2 u v)r^2+(9 u+6 v)r^3+(10+6 u v)r^4+(9 u+3 v)r^5}{+(2+2 u v)r^6+ur^7}}{ (1 + r u) (1 + r^2 + r v)(1 -r^4)}.
\end{equation}
From above inequality, it follows that
\begin{equation}\label{Realpart-K2}
\operatorname{Re}\left(\frac{zf'(z)}{f(z)}\right)  \geq
\frac{1-(7 + u v)r^2 -(8u+6v)r^3-
(9 + 6 u v)r^4-2(4u+v)r^5-(1+uv)r^6}{(1 - r^4) (1 + r u) (1 + r^2 + r v)}.
\end{equation}

\begin{enumerate}[(i)]
\item Set $x(r):=1-(u+v)r-(15+3 u v)r^2-(17 u+12 v)r^3-(17+12 u v)r^4-(15 u+3 v)r^5-(1+u v)r^6+ur^7$. Observe that $x(0)=1>0$ and $x(1)=-16 (2 + 2 u + uv + v) <0$ and thus the intermediate value theorem shows that a root of the equation (\ref{K2-S-P}) lies in $(0,1)$, denoted by $\rho_1$.

When $1/2 < a \leq 3/2$, in view of \cite[Lemma 2.2]{MR1415180}, the disc (\ref{discforK2}) is contained in the parabolic region $\{w \in \mathbb{C}: |w-1| < \operatorname{Re}w\}$ if
\[\frac{\splitfrac{(u+v)r+(8+2 u v)r^2+(9 u+6 v)r^3+(10+6 u v)r^4+(9 u+3 v)r^5}{+(2+2 u v)r^6+ur^7}}{ (1 + r u) (1 + r^2 + r v)(1 -r^4)} \leq \frac{1+r^4}{1-r^4}-\frac{1}{2}.\]
Equivalently,
\[\operatorname{Re}\left(\frac{zf'(z)}{f(z)}\right) > \left|\frac{zf'(z)}{f(z)}-1\right|\]
for $0<r\leq \rho_1$.
Thus, the radius of parabolic starlikeness for the class $\mathcal{K}^{3}_{b,c}$ is the number $\rho_1$.

\item For $0 \leq \alpha <1$, let $\rho_2 \in (0,1)$ be the smallest positive real root of the equation (\ref{K2-S-alpha}).
Thus, in view of (\ref{Realpart-K2}), we get
\[\operatorname{Re}\left(\frac{zf'(z)}{f(z)}\right) > \alpha\]
whenever $0<r \leq \rho_{2}$.


\item The number $\rho_3$ is the smallest real root of the equation in (\ref{K2-S-L}).
When $2\sqrt{2}/3 \leq a < \sqrt{2}$, in view of \cite[Lemma 2.2]{MR2879136}, the disc (\ref{discforK2}) lies in the lemniscate region $\{w \in \mathbb{C}: |w^2-1| < 1\}$ if
\[\frac{(u+v)r+(8+2uv)r^2 +(8u+5v)r^3+(4+4 u v)r^4+3ur^5}{(1 + r u) (1 + r^2 + r v)(1 - r^2)} \leq \sqrt{2}-1.\]
Hence, for $0<r\leq \rho_3$, we have
\[\left|\left(\frac{zf'(z)}{f(z)}\right)^2-1\right|<1.\]
the radius of lemniscate starlikeness for the class $\mathcal{K}^{3}_{b,c}$ is the number $\rho_3$.

\item The number $\rho_4$ is the smallest real root of the equation in (\ref{K2-S-e}).
When $1/e < a \leq (e+1/e)/2$, in view of \cite[Lemma 2.2]{MR3394060}, $f \in S^{*}_{e}$ if
\[\frac{\splitfrac{(u+v)r+(8+2 u v)r^2+(9 u+6 v)r^3+(10+6 u v)r^4+(9 u+3 v)r^5}{+(2+2 u v)r^6+ur^7}}{ (1 + r u) (1 + r^2 + r v)(1 -r^4)} \leq \frac{1+r^4}{1-r^4}-\frac{1}{e}.\]
Thus, the disc (\ref{discforK2}) is contained in the region $\{w \in \mathbb{C}: |\log w| < 1\}$ for $0<r\leq \rho_4$.

\item The number $\rho_5$ is the smallest real root of the equation in (\ref{K2-S-c}).
Using \cite[Lemma 2.5]{MR3536076}, the function $f \in S^{*}_{c}$ if
\[\frac{\splitfrac{(u+v)r+(8+2 u v)r^2+(9 u+6 v)r^3+(10+6 u v)r^4+(9 u+3 v)r^5}{+(2+2 u v)r^6+ur^7}}{ (1 + r u) (1 + r^2 + r v)(1 -r^4)} \leq \frac{1+r^4}{1-r^4}-\frac{1}{3}.\]
Thereby, the disc (\ref{discforK2}) lies inside  $\phi_{c}(\mathbb{D})$, where $\phi_{c}(z)=1+(4/3)z+(2/3)z^2$, if $0<r\leq \rho_5$.

\item The number $\rho_6$ is the smallest real root of the equation in (\ref{K2-S-sin}).
When $1-\sin 1 < a \leq 1+\sin 1$, an application of \cite[Lemma 3.3]{MR3913990} gives that the function $f \in S^{*}_{\sin}$ if
\[\frac{\splitfrac{(u+v)r+(8+2 u v)r^2+(9 u+6 v)r^3+(10+6 u v)r^4+(9 u+3 v)r^5}{+(2+2 u v)r^6+ur^7}}{ (1 + r u) (1 + r^2 + r v)(1 -r^4)} \leq \sin 1 -\frac{2r^4}{1-r^4}.\]
Hence, the disc in (\ref{discforK2}) is contained in the region $\phi_{s}(\mathbb{D})$, where $\phi_{s}(z)=1+\sin z$, for $0 < r \leq \rho_6$.

\item The number $\rho_7$ is the smallest real root of the equation in (\ref{K2-S-lune}).
Considering $\sqrt{2}-1 < a < \sqrt{2}+1$ and using \cite[Lemma 2.1]{MR3718233}, the disc (\ref{discforK2}) is contained in the region $\{w \in \mathbb{C}: 2|w| > |w^2-1|\}$ provided
\[\frac{\splitfrac{(u+v)r+(8+2 u v)r^2+(9 u+6 v)r^3+(10+6 u v)r^4+(9 u+3 v)r^5}{+(2+2 u v)r^6+ur^7}}{ (1 + r u) (1 + r^2 + r v)(1 -r^4)} \leq \frac{1+r^4}{1-r^4}+1-\sqrt{2}.\]
For $0<r \leq \rho_7$, the function $f$ satisfies
\[2\left|\frac{zf'(z)}{f(z)}\right| > \left|\left(\frac{zf'(z)}{f(z)}\right)^2-1\right|.\]

\item The number $\rho_8$ is the smallest real root of the equation in (\ref{K2-S-R}).
when $2(\sqrt{2}-1) < a \leq \sqrt{2}$, by \cite[Lemma 2.2]{MR3496681} the function $f \in S^{*}_{R}$ if
\[\frac{\splitfrac{(u+v)r+(8+2 u v)r^2+(9 u+6 v)r^3+(10+6 u v)r^4+(9 u+3 v)r^5}{+(2+2 u v)r^6+ur^7}}{(1 + r u) (1 + r^2 + r v)(1 -r^4)} \leq \frac{1+r^4}{1-r^4}+2-2\sqrt{2}.\]
Note that for $0 < r \leq \rho_8$, the disc (\ref{discforK2}) is contained in the region $\phi_{0}(\mathbb{D})$, where $\phi_{0}(z):=1+(z/k)((k+z)/(k-z))$ and $k=1+\sqrt{2}$.


\item Let $\rho_{9}$ denote the smallest real root of the equation (\ref{K2-S-Ne}).
When $1 \leq a < 5/3$, in view of \cite[Lemma 2.2]{MR4190740}, $f \in S^{*}_{Ne}$ if
\[\frac{\splitfrac{(u+v)r+(8+2 u v)r^2+(9 u+6 v)r^3+(10+6 u v)r^4+(9 u+3 v)r^5}{+(2+2 u v)r^6+ur^7}}{ (1 + r u) (1 + r^2 + r v)(1 -r^4)} \leq \frac{5}{3}-\frac{1+r^4}{1-r^4}.\]
Thus, the disc (\ref{discforK2}) is contained in the region $\phi_{Ne}(\mathbb{D})$ for $0 < r \leq \rho_{9}$, where $\phi_{Ne}(z)=1+z-z^3/3$.

\item Let $\rho_{10}$ denote the smallest real root of the equation (\ref{K2-S-SG}).
Applying \cite[Lemma 2.2]{MR4044913} for $1 \leq a < 2e/(1+e)$ gives $f \in S^{*}_{SG}$ if
\[\frac{\splitfrac{(u+v)r+(8+2 u v)r^2+(9 u+6 v)r^3+(10+6 u v)r^4+(9 u+3 v)r^5}{+(2+2 u v)r^6+ur^7}}{ (1 + r u) (1 + r^2 + r v)(1 -r^4)} \leq \frac{2e}{1+e}-\frac{1+r^4}{1-r^4}.\]
Hence, the disc (\ref{discforK2}) is contained in the region $\phi_{SG}(\mathbb{D})$ for $0 < r \leq \rho_{10}$, where $\phi_{SG}(z)=2/(1+e^{-z})$.
\end{enumerate}
\end{proof}

\section{Conclusion}
Three classes of analytic functions satisfying some conditions involving the ratio $f/g$ were introduced.
Several sharp radii estimates were determined for these classes by constraining the second coefficients of the functions.
Various generalizations of the existing work in the same field are also discussed.
The technique used in this paper can be imitated in finding radius estimates for various classes of analytic functions involving fixed second coefficient.


\begin{thebibliography}{00}
\bibitem{MR2879136}R. M. Ali, N. K. Jain\ and\ V. Ravichandran, Radii of starlikeness associated with the lemniscate of Bernoulli and the left-half plane, Appl. Math. Comput. {\bf 218} (2012), no.~11, 6557--6565. MR2879136

\bibitem{MR3722703}R. M. Ali\ et al., Radius of starlikeness for analytic functions with fixed second coefficient, Kyungpook Math. J. {\bf 57} (2017), no.~3, 473--492. MR3722703

\bibitem{MR0251208}D. A. Brannan\ and\ W. E. Kirwan, On some classes of bounded univalent functions, J. London Math. Soc. (2) {\bf 1} (1969), 431--443. MR0251208

\bibitem{MR3531955} S. Bulut\ and\ S. Porwal, Radii problems of the certain subclasses of analytic functions with fixed second coefficients, Electron. J. Math. Anal. Appl. {\bf 5} (2017), no.~1, 81--87. MR3531955

\bibitem{MR3913990}N. E. Cho\ et al., Radius problems for starlike functions associated with the sine function, Bull. Iranian Math. Soc. {\bf 45} (2019), no.~1, 213--232. MR3913990

\bibitem{MR0708494}P. L. Duren, {\it Univalent functions}, Grundlehren der mathematischen Wissenschaften, 259, Springer-Verlag, New York, 1983. MR0708494

\bibitem{MR3718233}S. Gandhi\ and\ V. Ravichandran, Starlike functions associated with a lune, Asian-Eur. J. Math. {\bf 10} (2017), no.~4, 1750064, 12 pp. MR3718233


\bibitem{MR4044913}P. Goel\ and\ S. Sivaprasad Kumar, Certain class of starlike functions associated with modified sigmoid function, Bull. Malays. Math. Sci. Soc. {\bf 43} (2020), no.~1, 957--991. MR4044913

\bibitem{MR3496681}S. Kumar\ and\ V. Ravichandran, A subclass of starlike functions associated with a rational function, Southeast Asian Bull. Math. {\bf 40} (2016), no.~2, 199--212. MR3496681

\bibitem{MR4166151} S. K. Lee, K. Khatter\ and\ V. Ravichandran, Radius of starlikeness for classes of analytic functions, Bull. Malays. Math. Sci. Soc. {\bf 43} (2020), no.~6, 4469--4493. MR4166151

\bibitem{MR1343506}W. C. Ma\ and\ D. Minda, A unified treatment of some special classes of univalent functions, in {\it Proceedings of the Conference on Complex Analysis (Tianjin, 1992)}, 157--169, Conf. Proc. Lecture Notes Anal., I, Int. Press, Cambridge, MA. MR1343506

\bibitem{MR0298014} C. P. McCarty, Functions with real part greater than $\alpha $, Proc. Amer. Math. Soc. {\bf 35} (1972), 211--216. MR0298014

\bibitem{MR3394060}R. Mendiratta, S. Nagpal\ and\ V. Ravichandran, On a subclass of strongly starlike functions associated with exponential function, Bull. Malays. Math. Sci. Soc. {\bf 38} (2015), no.~1, 365--386. MR3394060

\bibitem{MR0377031}Z. Nehari, {\it Conformal mapping}, reprinting of the 1952 edition, Dover Publications, Inc., New York, 1975. MR0377031

\bibitem{MR3419845}R. K. Raina\ and\ J. Sok\'{o}\l, Some properties related to a certain class of starlike functions, C. R. Math. Acad. Sci. Paris {\bf 353} (2015), no.~11, 973--978. MR3419845

\bibitem{MR0070715}M. O. Reade, On close-to-convex univalent functions, Michigan Math. J. {\bf 3} (1955), 59--62. MR0070715

\bibitem{MR0783568}M. S. Robertson, Certain classes of starlike functions, Michigan Math. J. {\bf 32} (1985), no.~2, 135--140. MR0783568

\bibitem{MR1128729}F. R\o nning, Uniformly convex functions and a corresponding class of starlike functions, Proc. Amer. Math. Soc. {\bf 118} (1993), no.~1, 189--196. MR1128729

\bibitem{MR1415180}T. N. Shanmugam\ and\ V. Ravichandran, Certain properties of uniformly convex functions, in {\it Computational methods and function theory 1994 (Penang)}, 319--324, Ser. Approx. Decompos., 5, World Sci. Publ., River Edge, NJ. MR1415180

\bibitem{MR3536076}K. Sharma, N. K. Jain\ and\ V. Ravichandran, Starlike functions associated with a cardioid, Afr. Mat. {\bf 27} (2016), no.~5-6, 923--939. MR3536076

\bibitem{MR1473947}J. Sok\'{o}\l\ and\ J. Stankiewicz, Radius of convexity of some subclasses of strongly starlike functions, Zeszyty Nauk. Politech. Rzeszowskiej Mat. No. 19 (1996), 101--105. MR1473947

\bibitem{MR4190740}L. A. Wani\ and\ A. Swaminathan, Starlike and convex functions associated with a nephroid domain, Bull. Malays. Math. Sci. Soc. {\bf 44} (2021), no.~1, 79--104. MR4190740

\bibitem{Mcgregor}T. H. MacGregor, The radius of univalence of certain analytic functions, Proc Amer. Math. Soc. {\bf 14} (1963), 514--520.

\bibitem{MR0148892}T. H. MacGregor, The radius of univalence of certain analytic functions. II, Proc Amer. Math. Soc. {\bf 14} (1963), 521--524. MR0148892


\end{thebibliography}
\end{document}